\documentclass[10pt, a4paper, twoside,reqno]{amsart}

\usepackage{fancyhdr}
\usepackage{graphicx}
\usepackage{amsmath,amssymb,amsfonts,amssymb}
\usepackage{mathtools, mathrsfs}
\usepackage{parskip}
\usepackage[latin1]{inputenc}
\usepackage{float}            
\usepackage{enumerate}
\usepackage[normalem]{ulem}
\usepackage{cancel}
\usepackage[all]{xy}
\usepackage{eurosym}
\usepackage{booktabs}

\usepackage{subfigure}

\input{commands.sty}
\input{commands_Peyman.sty}

\allowdisplaybreaks

\title{A Decomposition Method for Large Scale MILPs, with Performance Guarantees and a Power System Application}

\author[R. Vujanic]{Robin Vujanic} 
\author[P. Mohajerin Esfahani]{Peyman Mohajerin Esfahani} 
\author[P. Goulart]{Paul Goulart} 
\author[S. Mari\'ethoz]{S\'ebastien Mari\'ethoz} 
\author[M. Morari]{Manfred Morari}

\thanks{RV, PME, and MM are with the Automatic Control Laboratory, ETH Zurich, Switzerland, {\tt \{vujanicr,mohajerin,morari\}@control.ee.ethz.ch}. PG is with the Department of Engineering Science, University of Oxford, and SM is with Bern University of Applied Sciences.}

\begin{document}
\maketitle

 \begin{abstract}
 Lagrangian duality in mixed integer optimization is a useful framework for problems decomposition and for producing tight lower bounds to the optimal objective, but in contrast to the convex counterpart, it is generally unable to produce optimal solutions directly. In fact, solutions recovered from the dual may be not only suboptimal, but even infeasible. In this paper we concentrate on large scale mixed--integer programs with a specific structure that is of practical interest, as it appears in a variety of application domains such as power systems or supply chain management.
 We propose a solution method for these structures, in which the primal problem is modified in a certain way, guaranteeing that the solutions produced by the corresponding dual are feasible for the original unmodified primal problem. The modification is simple to implement and the method is amenable to distributed computations. We also demonstrate that the quality of the solutions recovered using our procedure improves as the problem size increases, making it particularly useful for large scale instances for which commercial solvers are inadequate. We illustrate the efficacy of our method with extensive experimentations on a problem stemming from power systems.
 \end{abstract}

\section{Introduction}
\label{sec:1_intro}
In this paper we investigate mixed-integer optimization problems in the form
\begin{equation}
\left\{ 
\begin{array}{lll}
	\underset{x}{\text{minimize}} & \sum\limits_{i \in I} c_i\tr x_i\\
	\text{subject to} & \sum\limits_{i \in I} H_i x_i \preceq b\\
	& x_i \in X_i &  \forall i \in I.
\end{array}
\right.
\tag{$\mathcal{P}$}
\label{eq:P}
\end{equation}
We refer to $b \in \mathbb{R}^m$ as the \emph{resource vector}, and to the sets $X_i$ as the \emph{subsystems}. We assume that each of the sets $X_i$ is a non-empty, compact, mixed-integer polyhedral set that can be written as
\begin{equation*}
		X_i = \left\{ x \in \mathbb{R}^{r_i} \times \mathbb{Z}^{z_i} \  \bigl|\bigr. \ A_ix \preceq d_i \right\},
		\label{polyhedral_mixed_integer_feasible_set}
\end{equation*}
with $A_i \in \mathbb{R}^{m_i \times n_i}$ and $d_i \in \mathbb{R}^{m_i}$. We further assume that the problem \ref{eq:P} is feasible and that the total number of subsystems $|I|$ is greater than the length $m$ of the resource vector. Our principal interest is in large-scale optimization problems, i.e.\ those for which $|I| \gg m$, while remaining finite. 

Problem \ref{eq:P} can be viewed generically as modeling any problem for which a large number of subproblems defined on the domains $X_i$, whose description can include integer variables, are coupled through a small number of complicating constraints $\sum_{i \in I} H_i x_i \preceq b$. These coupling constraints determine the limits on the available resources to be shared among the subsystems. Simple examples of problems in this form include classical combinatorial programs such as the multidimensional knapsack problem, in which $X_i = \zo$, and $c_i \succeq 0$, $H_i \succeq 0$ \cite{knapsack_heristics_and_applications}. 

More complicated instances of problems in the form \ref{eq:P}, with more detailed models for the subsystems $X_i$, arise in a variety of contexts. In power systems, scheduling operations of power generation plants \cite{yamin_review_2004} is a decision problem in which the subsystems are the generating units, integer variables in the local models arise due to, e.g., start-up and shut-down costs, and the coupling constraints are related to the requirement that generation must match load.  
In supply chain management, models fitting \ref{eq:P} appear in the problem of partial shipments \cite{partial_shipment,my_supply_chains}. Portfolio optimization for small investors, for which mixed-integer models have been proposed, is another example application \cite{small_investor_portfolio}. 
Finally, some sparse problems that do not naturally possess the structure of \ref{eq:P} can be reformulated to fit our framework by appropriately permuting rows and columns of the constraints matrix;  \cite{bergner_reformulation} proposes a method to automate this procedure.

A direct solution of \ref{eq:P} is typically problematic when the problem is very large, since the problem amounts to a mixed-integer linear program of possibly very large size.  As a result, the Lagrange dual of \ref{eq:P} is often taken as a useful alternative, because the resulting dual problem is separable in the subsystems despite the presence of the complicating constraints.  When this dual problem is solved by an iterative method, e.g.\ using the subgradient method \cite{bertsekas_nonlinear}, a candidate (primal) solution to \ref{eq:P} can be computed at each iteration.

One of the major drawbacks of this approach is that, for problems affected by a non-zero duality gap such as \ref{eq:P}, any guarantee about the properties of these candidate primal solutions is lost. Even at the dual optimal solution, the associated candidate primal solutions 
may be suboptimal and can even be infeasible. 

The principal goal of this paper is to propose a new solution method for problem \ref{eq:P} that preserves the attractive features of solution via the Lagrange dual, while at the same time protecting the recovered primal solutions from infeasibilty.

\textbf{Literature.} Lagrangian relaxation for mixed integer programs was first introduced by \cite{held_karp_1970}, and many of its theoretical properties were described in \cite{geoffrion_74}. Properties of the inner solutions in the convex case are well known \cite[Thm.\ 28.1]{rockafellar}. It is also well known that in general these properties are lost in the mixed-integer case \cite[Section 5.5.3]{bertsekas_nonlinear}. Because of this, primal recovery methods based on Lagrangian duality are often two-phase schemes in which an infeasible solution is found through duality in the first stage, and in the second stage it is rectified into a feasible one using heuristics, see, e.g., \cite{bertsekas_UC, conejo_redondo_shortterm}.

Duality for problems specifically in the form \ref{eq:P} has been studied at least as early as in \cite{ekeland}, where some of its special features were first characterized. In particular, it was noted that the duality gap for this program structure decreases in relative terms as the problem increases in size, as measured by the cardinality of $I$. We will show that the mechanism behind this vanishing gap effect can also be used to recover ``good'' primal solutions for the mixed-integer program \ref{eq:P} directly from the dual, in a way that resembles the convex (zero gap) case. 

In practical applications, this behaviour of the duality gap has been observed in \cite{bertsekas_UC} in the context of unit commitments for power systems. In this case it is exploited in an algorithm that provides solutions to the extended master problem, but no connection to the solutions of the inner problem is provided. It also appears in the multistage stochastic integer programming literature \cite{vanishing_gap_in_sips,caroe_schultz}, where it is used to gauge the strength of the Lagrangian relaxation, but in which no relations to primal solutions are drawn. Another domain in which diminishing gap has been used is in communications, more precisely in optimization of multicarrier communication systems \cite{duality_in_communications}. However, in this case non-convexity is in the objective function rather than due to the presence of integer variables.

\textbf{Current Contribution.} In this paper we further investigate duality for programs structured as \ref{eq:P} and focus on the primal solutions recovered at the dual optimum.
\begin{itemize}
	\item We provide a new relation between the optimizers of a convexified form of \ref{eq:P} and the solutions obtained from the dual problem. This relation holds under mild conditions that are commonly satisfied in practice.
	\item In light of this relation, we propose a new solution method that is guaranteed to produce feasible solutions. 
	\item We also provide a performance bound of the solutions recovered, which indicates that their quality improves as the problem size increases. For particular structures, arising e.g. from underlying physical networks, we refine our theoretical results to improve the performance of the method.
\end{itemize}
From a practical point of view, we note that our proposed procedure is straightforward to implement and is amenable to distributed computations. The performance bound indicates that the method is particularly attractive for the larger problem instances, for which generic purpose solvers may be inadequate.
We show that the theoretical results are effective in practice via extensive numerical experiments on difficult problems stemming from the field of power systems control. Our method substantially outperforms commercial solvers on these problems. The limitations of the proposed method, as well as ideas to mitigate them, are also discussed in the paper.


\textbf{Structure of the Paper.} The paper is structured as follows: in Section \ref{sec:2_duality_for_P} we review some of the known results concerning duality for the specific structure of \ref{eq:P}, and we provide a new result related to the primal solutions recovered from the dual. In Section \ref{sec:3_robust_feasible_soln} we propose a new method for primal solution recovery, and provide  performance bounds for these solutions. We also give some results on how to further improve the solutions' quality in some special cases. In Section \ref{sec:4_example} we verify the efficacy of our proposed method on a difficult optimization problem stemming from power systems, and in Section \ref{sec:5_conclusion} we conclude the paper.

\textbf{Notation.} Given some optimization problem $\mathcal{A}$, we denote with $J_{\mathcal{A}}^\star$ its optimal objective and with $J_{\mathcal{A}}(x)$ the performance of the solution $x$ with respect to the objective of $\mathcal{A}$. For a given set $X$, we denote by $\conv(X)$ its convex hull and by $\vert(X)$ the set of vertices of $\conv(X)$. With ``$\succeq$'' we always intend component-wise inequalities (between vectors or matrices), and with $\otimes$ we indicate the cartesian product of sets.
The support of a vector $\mathrm{supp}(x)$ is the set of indexes of the non-zero elements: $\mathrm{supp}(x) = \{i: x_i \neq 0\}$, while $(x)^+$ is the projection of $x$ onto the positive orthant, i.e., $(x)^+ \doteq \max (0, x)$. For the specific structure of \ref{eq:P}, we use the \emph{overbar} symbol to indicate quantities related to the contracted version of \ref{eq:P}, as introduced in Section \ref{sec:3_robust_feasible_soln}. Thus, for instance, \ref{eq:Pbar} is the contracted form of \ref{eq:P} and $\Dbar$ is its dual. We use parenthesis to avoid confusing the sub- and superscripts, e.g., we denote by $(x_{\mathcal{P}})_i$ the part of $x_{\mathcal{P}}$ related to subproblem $i \in I$ of problem \ref{eq:P}. Finally, we use the superscript $H^k$ to denote the $k$--th row of matrix $H$.

\section{Duality for Problem \ref{eq:P}}
\label{sec:2_duality_for_P}


Consider the dual function $d: \mathbb{R}^m \rightarrow \mathbb{R}$ of problem \ref{eq:P}, defined as
\begin{equation*}
	d(\lambda) \eqdef \min_{x \in X} \bigg ( \sum_{i \in I} c_i\tr x_i + \lambda\tr(\sum_{i \in I} H_i x_i - b) \bigg),
	\label{eq:dual_function}
\end{equation*}
and then associate to this function the optimization problem 
\begin{equation}
\left\{ 
\begin{array}{ll}
	\sup\limits_{\lambda} & - \lambda\tr b + \sum\limits_{i \in I} \min\limits_{x_i \in X_i} \left( c_i\tr x_i  + \lambda\tr H_i x_i \right)\\
	\text{s.t.} & \lambda \succeq 0.
\end{array}
\right.
\label{eq:D}
\tag{$\mathcal{D}$}
\end{equation}
We call \ref{eq:D} the \emph{dual problem} of \ref{eq:P}, and we refer collectively to the minimizations within \ref{eq:D}, i.e.,
\begin{equation}
	\min_{x_i \in X_i} \left( c_i\tr x_i  + \lambda\tr H_i x_i \right),
	\label{eq:inner_problem_sdef}
\end{equation}
as the \textit{inner problem}. There is substantial practical interest in understanding the properties of the solutions to the inner problem \eqref{eq:inner_problem_sdef} because they are obtained by solving $|I|$ independent (and lower dimensional) minimization problems, in contrast to the single large coupled problem \ref{eq:P}. Additionally, they are usually obtained as by-products of methods used to solve \ref{eq:D} (e.g.\ the subgradient method). These solutions, in particular those attained at the vertices of $\conv(X_i)$, are the central object of this paper:

\begin{Def}[inner problem solutions]
	\label{definition_solns_to_inner_problem}
	For a given multiplier $\lambda \succeq 0$, the set \mbox{$\mxi \subseteq \mathbb{R}^{n_i}$} is defined as the set of inner solutions that are attained at the vertices of $X_i$, i.e.\
	\begin{equation}
		\mxi \eqdef \vert(X_i) \cap \arg \min_{x_i \in X_i} \left( c_i\tr x_i  + \lambda\tr H_i x_i \right).
	\label{eq:mathcal_xi}
	\end{equation}
	Furthermore, we denote by $\xl$ any selection from the set $\mx$, and refer to it as an inner solution.
\end{Def}
\begin{Fact}
\label{prop:there_are_vertex_solns}
The sets $\mxi$, $i \in I$, are non-empty for any $\lambda \succeq 0$.
\end{Fact}
\begin{proof}
See Appendix \ref{sec:proof_fact_1}.
\end{proof}

\subsection{Bound on Duality Gap}

For a general mixed integer linear program, the inner solutions $\xls \in \mxs$, in which $\ls$ is an optimizer of \ref{eq:D}, do not possess any ``nice'' property in general: they can be non-unique, suboptimal and even infeasible. In this paper we show that inner solutions for programs structured specifically as \ref{eq:P} do acquire some useful properties. Informally speaking, these additional properties arise mainly from the fact that, as \ref{eq:P} grows in size, it tends to closely approximate a convex program. One known result of this is that the duality gap between \ref{eq:P} and \ref{eq:D} vanishes, in relative terms, as $|I|$ increases.

\begin{Theorem}[bound on duality gap]
\label{thm:bound_on_duality_gap}
	Assume that for any $x_i \in \conv(X_i)$, there exists an $\tilde{x}_i \in X_i$ such that $H_i \tilde{x}_i \preceq H_i x_i$. Then
	\begin{equation}
	\label{eq:duality_gap_bound}
		J_{\mathcal{P}}^\star - J_{\mathcal{D}}^\star \leq m \cdot \underset{i \in I}{\max} \ \gamma_i, \quad \gamma_i \Let \underset{x_i \in X_i}{\max} c_i\tr x_i - \underset{x_i \in X_i}{\min} c_i\tr x_i.
	\end{equation}
\end{Theorem}
In consideration of Theorem \ref{thm:bound_on_duality_gap}, let $|I|$ increase, while $m$ remains constant and the sets $\left\{X_i \right\}_{i \in I}$ are uniformly bounded. If $\JP$ increases linearly with $|I|$, then 
\begin{equation}
\label{eq:vanishing_duality_gap}
	\frac{J_{\mathcal{P}}^\star - J_{\mathcal{D}}^\star}{J_{\mathcal{P}}^\star} \rightarrow 0 \quad \text{as} \quad |I| \rightarrow \infty.
\end{equation}
An early proof of this result appears in \cite{ekeland}, while a more recent version is in \cite[Prop.\ 5.26, p.\ 374]{bertsekas_cvx_multiplier}. The same result also holds for more general problems; see \cite[Prop.\ 5.7.4, p.\ 223]{bertsekas_cvx_theory}.

Note that while Theorem \ref{thm:bound_on_duality_gap} ensures the existence of a primal feasible solution satisfying the performance bound \eqref{eq:duality_gap_bound}, it does not provide an algorithmic way to produce it.  Furthermore, the assumption required by Theorem \ref{thm:bound_on_duality_gap} is restrictive; an example that does not fulfil this assumption is discussed in Section \ref{sec:4_example}, see Remark \ref{remark:old_result_not_ok_with_EV}. In this work we lift this assumption, at the cost of conservatism and thus performance of the solutions recovered.



\subsection{Geometric Properties of the Inner Solutions $\xls$}
\label{sec:subsec_2_1}


Here we present a new connection between the inner solutions $\xls$ and the optimizers of the following optimization program
\begin{equation}
\left\{ 
\begin{array}{lll}
	\underset{x}{\text{minimize}} & \sum\limits_{i \in I} c_i\tr x_i\\
	\text{subject to} & \sum\limits_{i \in I} H_i x_i \preceq b\\
	& x_i \in \conv(X_i) &  \quad \forall i \in I,
\end{array}
\right.
\label{eq:P_lp}
\tag{$\mathcal{P}_{\mathrm{LP}}$}
\end{equation}
which amounts to a linear program. We denote by $\JPLP$ its optimal value, and by $\xslp$ one of its optimizers. The relaxation \ref{eq:P_lp} plays a central role in Lagrangian duality for mixed integer programs;
it is in fact well known that \ref{eq:P_lp} satisfies the (non-obvious) relation $\JPLP = \JD$ \cite[Thm. 1b, p.87]{geoffrion_74}. Accordingly, \ref{eq:P_lp} is often used to gain insight into the strength of the relaxation, i.e., the tightness of the lower bounds to $\JP$ provided by the Lagrangian dual. While in most practical cases one cannot solve \ref{eq:P_lp} directly since an explicit description of the polyhedral sets $\conv(X_i)$ is required, column generation techniques construct approximations of \ref{eq:P_lp} \cite{branch_and_price,desrosiers_col_gen_primer,vanderbeck_col_gen}. It must be further emphasized that even though \ref{eq:P_lp} is a relaxation of \ref{eq:P} and is a linear program, it does \textit{not} coincide with the standard linear relaxation in which the integrality constraints on the discrete variables are relaxed to intervals. In fact, \ref{eq:P_lp} is usually tighter; see \cite[Thm. 1a]{geoffrion_74}. 



In consideration of the Shapley--Folkman--Starr theorem \cite[p.233]{ekeland}, one can expect the vertices of the convexified problem \ref{eq:P_lp} to have ``structure'', i.e.\ for $\xislp$ to belong to $X_i$ for at least $|I|- m -1$ subproblems, and $\xislp \in \conv(X_i)$ for the remaining $m+1$ ones, see \cite[Thm. 1]{my_supply_chains}. This number can be improved to $|I|-m$ using an argument based on simplex tableaus instead of the Shapley--Folkman--Starr theorem. We use this tighter version here, and in the following new result, the crucial technical theorem of the paper, we extend it by establishing that the subproblems for which $\xislp \in X_i$ also ``freeze'' the corresponding inner solutions $\xis$.

\begin{As}[uniqueness for \ref{eq:P_lp} and \ref{eq:D}]
\label{assumption:uniqueness}
The programs \ref{eq:P_lp} and \ref{eq:D} have unique solutions $\xslp$ and $\ls$, respectively.
\end{As}

\begin{Theorem}[relation between $\xslp$ and $\xs$]
\label{thm:x_lp-x_lambda}
Under Assumption \ref{assumption:uniqueness}, the solutions $\xslp$ and $\xs$ differ in at most $m$ subproblem components, for any selection of $\xs \in \mxs$. That is, for all $\xs \in \mxs$ there exists $I_1 \subseteq I$, with $|I_1| \geq |I|-m$, such that $\xis = \xislp$.
\label{thm_1}
\end{Theorem}
\begin{proof}
See Appendix \ref{sec:proof_connectedness}.
\end{proof}


Assumption \ref{assumption:uniqueness} concerns two linear programs (see program \ref{eq:D_lp} in Section \ref{sec:proof_connectedness} for the LP version of \ref{eq:D}). Uniqueness of primal and dual optimizers in the linear programming case is discussed in \cite{mangasarian}, where necessary and sufficient conditions are provided. There are degenerate cases in which this assumption may fail, in particular when the problem's data is affected by a high degree of symmetry. 
These cases, however, can always be avoided by adding negligible perturbations to the cost and resource vectors.

Furthermore, note that while the structural properties of $\xslp$ appeared in the literature \cite{bertsekas_UC,bertsekas_cvx_theory,my_supply_chains}, the contribution here is to ensure that, under Assumption \ref{assumption:uniqueness}, these advantageous properties are transferred to the inner solutions $\xs$. This is of substantial practical interest, because it is the inner solutions that one has direct access to when solving the dual. In the following we provide an analytical example that further illustrates the significance of Theorem \ref{thm:x_lp-x_lambda}. It also includes a counterexample, showing how the desired assertion may fail in absence of Assumption \ref{assumption:uniqueness}.

\begin{Example}
\label{example:example_1}
Suppose we have to
\begin{equation}
\left\{ \begin{array}{lll}
	\text{minimize} & \sum_{i=1}^4 c_i x_i\\
	\text{s.t.} & \sum_{i=1}^4 H_i x_i \leq 11.1\\
	& x_i \in X_i & i = 1,\dots,4
\end{array}\right.
\label{example_small_problem_2}
\end{equation}
and $X_i = \left\{ x \in \mathbb{Z}_+^2 \left| \  A_ix_i \leq d_i \right.\right\}$ with
\begin{equation*}
\begin{array}{llll}
	A_1 = \begin{bmatrix} 0 & 1\\ 1 & 1 \end{bmatrix} & d_1 = \begin{bmatrix} 1.2\\ 2.1\end{bmatrix} & c_1 = [1,1] & H_1 = [1,1]
	\vspace{0.5mm}\\
	A_2 = \begin{bmatrix} 0 & 1\\ 1 & 0 \end{bmatrix} & d_2 = \begin{bmatrix} 0.6\\ 2.1\end{bmatrix} & c_2 = [-2,1] & H_2 = [5,1]
	\vspace{0.5mm}\\
	A_3 = \begin{bmatrix} 1 & 0\\ -0.5 & 1 \end{bmatrix} & d_3 = \begin{bmatrix} 2.2\\ 1.1\end{bmatrix} & c_3 = [0.5,-1] & H_3 = [1,1]
	\vspace{0.5mm}\\
	A_4 = \begin{bmatrix} 1 & 0\\ 0 & 1 \end{bmatrix} & d_4 = \begin{bmatrix} 1.2\\ 2\end{bmatrix} & c_4 = [-3,0.5] & H_4 = [1,1],
\end{array}
\end{equation*}
see Figure \ref{fig:appendix_example_1}. Relaxing the constraint $\sum_{i= 1}^4 H_i x_i \leq 11.1$ in this problem leads to the dual function
\begin{figure}
	\centering
	\includegraphics[width=1\linewidth]{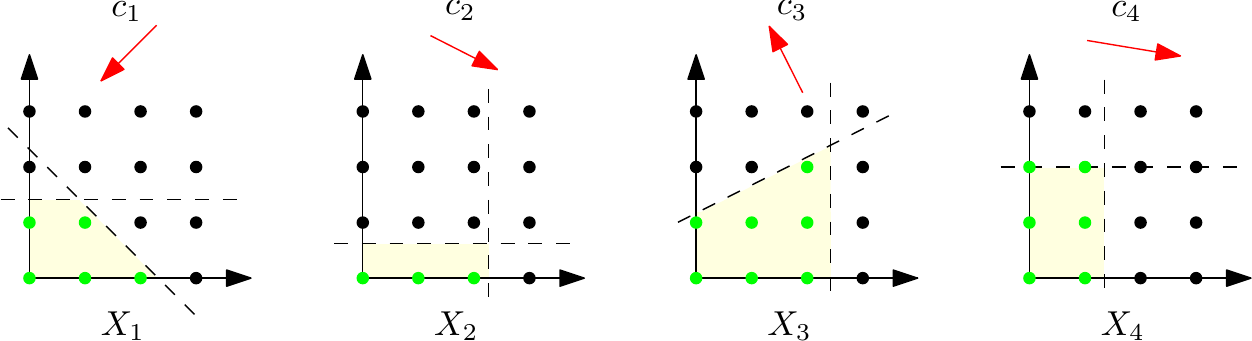}
	\caption{Illustration of the sets $X_i$ in Example \ref{example:example_1}.}
	\label{fig:appendix_example_1}
\end{figure}
\begin{equation*}
	d(\lambda) = \left\{\begin{array}{ll} 
	-8 + 0.9 \lambda \quad & 0 \leq \lambda \leq 2/5\\ 
  - 4 - 8.9 \lambda & 2/5 < \lambda \leq 1\\ 
	-3 - 9.9 \lambda & 1 < \lambda \leq 3\\
	- 10.9 \lambda & \lambda > 3,
	\end{array}\right.
\end{equation*}
so that the dual optimizer is $\lambda^* = 2/5$, and $d(\ls) = \JD = -7.64$, while the primal optimal objective is $\JP = -7$ (note the duality gap). The corresponding sets of inner solutions are, according to \eqref{eq:mathcal_xi},
\begin{equation*}
\begin{array}{l}
	\mathcal{X}_1(\ls) = \left\{ \begin{bmatrix} 0\\0\end{bmatrix} \right\} \quad
	\mathcal{X}_2(\ls) = \left\{ \begin{bmatrix} 0\\0\end{bmatrix},
	\begin{bmatrix} 2\\0\end{bmatrix} \right\} 
	\\
	\vspace{-3.5mm}\\
	\mathcal{X}_3(\ls) = \left\{ \begin{bmatrix} 0\\1\end{bmatrix} \right\} \quad
	\mathcal{X}_4(\ls) = \left\{ \begin{bmatrix} 1\\0\end{bmatrix} \right\}.
\end{array}
\end{equation*}
On the other hand, $\xslp$ is unique and is given by
\begin{equation*}
\begin{array}{c}
	(\xslp)_1 = \begin{bmatrix} 0\\0\end{bmatrix} \quad (\xslp)_2 = \begin{bmatrix} 1.82\\0\end{bmatrix} \quad
	(\xslp)_3 = \begin{bmatrix} 0\\1\end{bmatrix} \quad (\xslp)_4 = \begin{bmatrix} 1\\0\end{bmatrix}.
\end{array}
\end{equation*}
Notice how the relationship $\xis = \xislp$ holds for $i \in \left\{1,3,4\right\} = I_1$, and that the cardinality of $|I_1|$ satisfies $|I_1| \geq |I| - m = 4 - 1 = 3$. The validity of Theorem \ref{thm:x_lp-x_lambda} is thus verified.

On the other hand, to see how the Theorem may fail in absence of Assumption \ref{assumption:uniqueness}, consider again problem \eqref{example_small_problem_2}, but now with $b = 6$ and with the subsystems determined by
\begin{eqnarray*}
\begin{array}{llll}
	A_i = \begin{bmatrix} 1 & 0\\ 0 & 1 \end{bmatrix} \quad d _i = \begin{bmatrix} 3.2\\ 1.4\end{bmatrix} \quad c_i = [-1,1] \quad H_i = [1,1]
\end{array}
\end{eqnarray*}
for $i= 1,\dots,4$.
Notice that all the subsystems are identical, hence the problem is highly symmetric. The dual function in this case is
\begin{equation*}
	d(\lambda) = \left\{\begin{array}{ll} 
	-12 + 6\lambda & 0 \leq \lambda \leq 1\\ 
  	-6 \lambda & \lambda > 1,
	\end{array}\right. 
\end{equation*}
and the unique dual optimizer is $\ls = 1$. However, $\xslp$ is not unique. For example
\begin{equation*}
	(\xslpb)_1 = \begin{bmatrix} 0.7\\0\end{bmatrix} \quad (\xslpb)_2 = \begin{bmatrix} 1.6\\0\end{bmatrix} \quad (\xslpb)_3 = \begin{bmatrix} 0.6\\0\end{bmatrix} \quad (\xslpb)_4 = \begin{bmatrix} 3.1\\0\end{bmatrix}\\
\end{equation*}
and
\begin{equation*}
	(\xslpb)_1 = \begin{bmatrix} 2\\0\end{bmatrix} \quad (\xslpb)_2 = \begin{bmatrix} 2\\0\end{bmatrix} \quad (\xslpb)_3 = \begin{bmatrix} 2\\0\end{bmatrix} \quad (\xslpb)_4 = \begin{bmatrix} 0\\0\end{bmatrix}\\
\end{equation*}
are both valid optimizers of $\mathcal{P}_{\mathrm{LP}}$. Assumption \ref{assumption:uniqueness} is therefore not fulfilled. The sets of inner solutions are
\begin{equation*}
	\mathcal{X}_i(\lsb) = \left\{ \begin{bmatrix} 0\\0\end{bmatrix},
	\begin{bmatrix} 3\\0\end{bmatrix}  \right\}, \quad i = 1,\dots,4,
\end{equation*}
and the relationship of Theorem \ref{thm:x_lp-x_lambda} is violated.

\end{Example}

\begin{Remark}[nonlinear extension]
Theorem \ref{thm_1} holds even when the objective and the coupling constraints functions are \textit{concave}. This is immediate by noticing that, in either case, local solutions are found at the vertices of $X_i$, according to a more general version of the Fundamental Theorem of Linear Programming, see \cite[Prop.\ 2.4.2]{bertsekas_cvx_theory}. The passage \eqref{eq:min} in the proof of Lemma \ref{prop:there_are_vertex_solns} remains unchanged, and the proof of Theorem \ref{thm_1} follows verbatim. 
\end{Remark}



\section{A Distributed Solution Method for \ref{eq:P}}
\label{sec:3_robust_feasible_soln}

Informally speaking, Theorem \ref{thm_1} says that the inner solutions $\xls$ nearly coincide with those of $\xslp$, with the cardinality of their difference bounded by $m$, i.e., the dimension of the coupling constraint. Since $\xslp$ is feasible with respect to the coupling constraints and attains a better objective than $\JP$, one can expect the solutions obtained from solving the dual to be nearly feasible and to attain good objective values. In this section we exploit this result to propose a method aimed at obtaining ``good'' feasible solutions to problem \ref{eq:P} in a distributed fashion. 


\subsection{Contraction of the Resources}


Our proposed method is to contract the resources vector $b$ by an appropriate amount, which is determined by the results of the previous section. We show that any inner solution recovered at the dual optimum $\ls$ of the contracted problem is a feasible solution for \ref{eq:P}. We also provide a performance bound for these solutions, which indicates that their quality improves with increasing problem size. 

Consider the following modified version of problem \ref{eq:P}
\begin{equation}
\left\{ 
\begin{array}{lll}
	\text{minimize} & \sum\limits_{i \in I} c_i\tr x_i\\
	\text{subject to} & \sum\limits_{i \in I} H_i x_i \preceq \bar{b}\\
	& x_i \in X_i &  \forall i \in I.
\end{array}
\right.
\tag{$\overline{\mathcal{P}}$}
\label{eq:Pbar}
\end{equation}
The resource vector $b$ has been contracted to $\bar{b} \eqdef b - \rho$, where the $k$-th element of the contraction $\rho \in \mathbb{R}^m$ is given by
\begin{equation}
  \label{rho} 
  \rho^k =  m \cdot \underset{i \in I}{\max} \left( \underset{x_i \in X_i}{\max} H^k_i x_i - \underset{x_i \in X_i}{\min} H^k_i x_i \right),
\end{equation}
where $H_i^k$ is the $k$-th row of $H_i$. Correspondingly, we introduce the problems $\overline{\mathcal{P}}_{\mathrm{LP}}$ and $\overline{\mathcal{D}}$, defined similarly to \ref{eq:P_lp} and \ref{eq:D}, replacing the resource vector $b$ with $\bar{b}$. We next establish that the primal solutions recovered from the dual of \ref{eq:Pbar} are feasible for \ref{eq:P}.
\begin{Theorem}[feasible solutions]
If Assumption \ref{assumption:uniqueness} holds for the programs $\overline{\mathcal{P}}_{\mathrm{LP}}$ and $\overline{\mathcal{D}}$, then any selection $\xsb \in \mxsb$ is feasible for \ref{eq:P}, where $\lsb$ is the optimal solution of $\overline{\mathcal{D}}$.
\label{theorem_main_tight}
\end{Theorem}
\begin{proof}
See Appendix \ref{sec:proof_contraction_works}.
\end{proof}
The method is easy to implement because the amount of contraction required usually necessitates only simple computations\footnote{Dual methods are most useful when the computation of the inner solutions is substantially easier than the coupled system. 
To compute the contraction, however, we have to perform \textit{maximizations} of the form $\max_{x_i \in X_i} H^k_i x_i$, which for mixed integer problems are not necessarily as easy as minimizations over the same feasible set.}, and these can be carried out in a distributed fashion. Furthermore, for the solution of the dual problem well established methods exist (e.g., the subgradient method) and they can be directly applied here. 

The critical assumption of Theorem \ref{theorem_main_tight} is that the resources available should be sufficiently abundant, such that the problem remains feasible after the contraction has been applied. In Section \ref{sec:reducing_cons} we discuss practical cases in which it is possible to safely decrease the necessary resource reduction.

In the next Theorem we assess the performance of the solutions $\xsb$. In order to obtain an explicit bound, we first make the following assumption.

\begin{As}[Slater point with increasing slack]
\label{ass:slater_point}
There exist $\zeta > 0 $ and $\wt x_i \in \conv(X_i)$ for all $i \in I$ such that 
\begin{equation}
\label{eq:slater_point_has_increasing_slask}
	\sum\limits_{i \in I} H_i \wt x_i \preceq \bar{b} - \zeta |I| \ind.
\end{equation}
\end{As}
\begin{Theorem}[performance guarantee]
\label{thm:solutions_quality}
Suppose that the programs $\overline{\mathcal{P}}_{\mathrm{LP}}$ and $\overline{\mathcal{D}}$ satisfy Assumption \ref{assumption:uniqueness} and Assumption \ref{ass:slater_point} holds. Then any solution $\xsb \in \mxsb$ recovered satisfies
\begin{equation}
	J_{\mathcal{P}}(\xsb) - \JP \leq (m + \|\rho\|_\infty/\zeta) \cdot \max_{i \in I} \gamma_i,
\label{eq:performance_bound}
\end{equation}
where $\gamma_i$ and $\rho$ are as defined in \eqref{eq:duality_gap_bound} and \eqref{rho}, respectively. 
\end{Theorem}
\begin{proof}
See Appendix \ref{sec:proof_quality_of_solns}.
\end{proof}
In view of Theorem \ref{thm:solutions_quality}, if the sets $\{X_i\}_{i \in I}$ are uniformly bounded and $\JP$ grows linearly in terms of $|I|$, then 
	\begin{equation}
	\label{eq:vanishing_opt_gap}
		\frac{J(\xsb) - \JP}{\JP} \rightarrow 0 \quad \text{as} \quad |I| \rightarrow \infty.
	\end{equation} 
Accordingly, the quality of the solutions recovered increases the larger the problem becomes, as the optimality gap decreases at a ``$1/|I|$'' rate. In Section \ref{sec:increasing_margin_discussion} we will discuss Assumption \ref{ass:slater_point} and show this asymptotic behavior can be expected even in the absence of a Slater point.

Theorem \ref{theorem_main_tight} and \ref{thm:solutions_quality} provide a systematic way to produce solutions that are guaranteed to be feasible and that satisfy the performance bound \eqref{eq:performance_bound}. This bound resembles \eqref{eq:duality_gap_bound}, where the additional term ``$\|\rho\|_\infty/\zeta$'' may be viewed as the price to ensure feasibility and to lift the assumption required by Theorem \ref{thm:bound_on_duality_gap}.



\subsection{Reducing Conservatism}
\label{sec:reducing_cons}

The contraction proposed in Theorem \ref{theorem_main_tight} can be interpreted as a robustification of problem \ref{eq:P} toward alterations of $m$ local solutions $x_i$. In this section we take a closer look at the coupling constraints matrix ${H \eqdef \left[ H_1, H_2, \dots, H_{|I|}\right]}$ and discuss some special cases in which its structure can be exploited to safely reduce the necessary contraction.


Suppose that the matrix $H$ has block structure, as depicted in Figure \ref{fig:coupling_matrix_block_structure_1}. As illustrated, we introduce the set $I_k$ as the index set of the subsystems contributing to the $k$-th coupling constraint, i.e., for which $H_i^k \neq 0$. We furthermore define the submatrix $[H_i]_{i \in I_k}$, obtained by collecting the columns of $H$ related to the subsystems in $I_k$. 

Such a block structured $H$ may arise in applications in which the resources present a hierarchical structure, or when the optimization is over tree or tree-star networks, as shown on Figure \ref{fig:coupling_matrix_block_structure_2}. In this case, the uniform contraction proposed in Theorem \ref{theorem_main_tight} can be safely reduced.

\begin{figure}
\label{fig:coupling_matrix_block_structure}
\centering
\subfigure[]{							\includegraphics[width=0.52\linewidth]{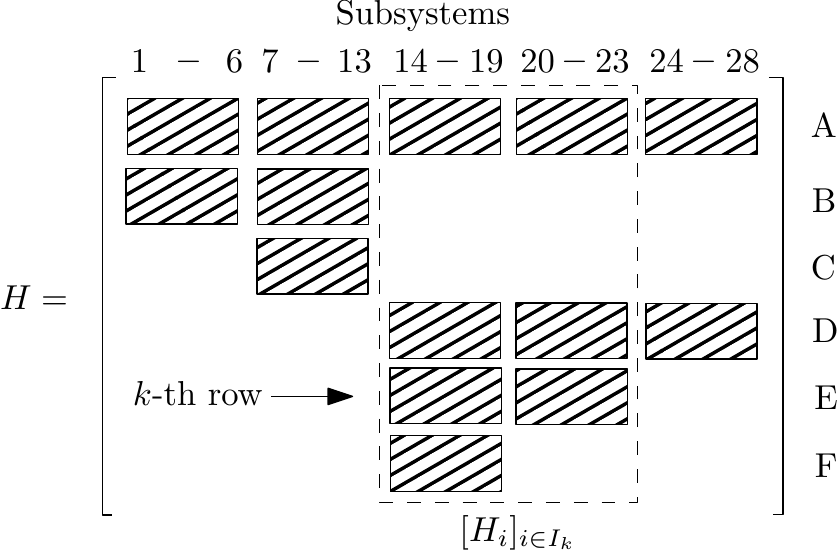}
\label{fig:coupling_matrix_block_structure_1}
} \quad
\subfigure[]{							\includegraphics[width=0.40\linewidth]{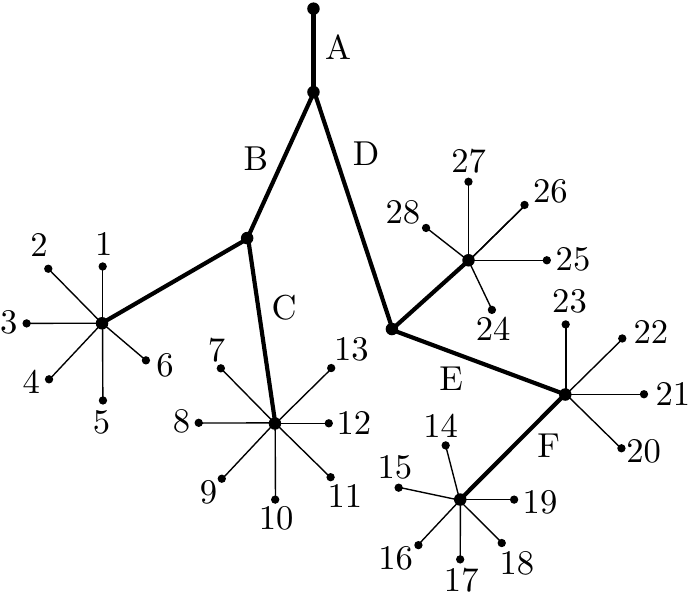}
\label{fig:coupling_matrix_block_structure_2}
}
\caption{(a) block structure considered in Theorem \ref{thm:reducing_contraction_with_rank_hi}; hatched boxes indicate non-zero submatrices, while the dashed box contains the submatrix $[H_i]_{i \in \I_k}$. (b) an example network that would give rise to such a block structured $H$. In this illustrative Figure there are 28 subsystems, and 6 sets of coupling constraints determined by constraints on the network links A--F.}
\end{figure}

\begin{Theorem}[refinement for block structure]
\label{thm:reducing_contraction_with_rank_hi}
Theorem \ref{theorem_main_tight} holds with the contraction \eqref{rho} substituted by
\begin{eqnarray}
  \label{rho_reduced_to_rank_hi} 
  \rho^k =  \rank([H_i]_{i \in \I_k}) \cdot \underset{i \in I_k}{\max} \left( \underset{x_i \in X_i}{\max} H^k_i x_i - \underset{x_i \in X_i}{\min} H^k_i x_i \right).
\end{eqnarray}
\end{Theorem}
\begin{proof}
See Appendix \ref{sec:proof_reduced_conservatism}.
\end{proof}

This theorem implies, as a special case, that we can generally substitute $m$ with $\rank(H)$ in \eqref{rho}, independently of whether the problem has block structure. This is important when the vectors determining the coupling constraints are linearly dependent. An example exploiting this result is discussed in Section \ref{sec:4_example}.

Furthermore, instead of immunizing against $\rank([H_i]_{i \in \I_k})$ times the largest subproblem budget consumption change, it is sufficient to immunize against the $\rank([H_i]_{i \in \I_k})$ largest ones, i.e.,
\begin{Remark}
The contraction \eqref{rho} can be safely substituted by
\begin{eqnarray}
	\rho^k =  \underset{\begin{subarray}{c}\tilde{I} \subseteq \I_k\\ |\tilde{I}| = \rank([H_i]_{i \in \I_k}) \end{subarray}}{\max} \left( \sum\limits_{i \in \tilde{I}} \underset{x_i \in X_i}{\max} H^k_i x_i - \underset{x_i \in X_i}{\min} H^k_i x_i \right).
\end{eqnarray}
\end{Remark}

Finally, an important subclass of problems for which we can suppress the necessary contraction to $\rho = 0$ is the following.

\begin{Remark}
If $H_i x_i \succeq 0$ for all $x_i \in X_i$, and $\textbf{0} \in X_i$, then one can obtain the same performance bound as in \eqref{eq:performance_bound} while setting $\rho = 0$, and a feasible solution can be recovered by setting $\xis = \textbf{0}$ for at most $m$ subsystem solutions.
\end{Remark}
This is for instance the case for the (multidimensional) knapsack problem and some of its variants. Namely, a feasible solution is obtained by removing at most $m$ items from the knapsacks.



\subsection{Further Discussion on the Performance Bound}
\label{sec:increasing_margin_discussion}

One of the key factors contributing to the optimality gap identified in Theorem \ref{thm:solutions_quality} is the performance loss due to the contraction $\rho$, determined by \mbox{$[J^\star_{\overline{\mathcal{P}}_{\mathrm{LP}}} - \JPLP]$};
see the proof of Theorem \ref{thm:solutions_quality}, in particular the term (ii), in Section \ref{sec:proof_quality_of_solns}. 
In Theorem \ref{thm:solutions_quality}, Assumption \ref{ass:slater_point} allows us to establish an explicit bound on this term. Here we show that this performance loss can be characterized by the data of only $m$ subsystems, which explains why one may expect a behavior for the optimality gap similar to \eqref{eq:vanishing_opt_gap} even in the absence of Assumption \ref{ass:slater_point}. 


		
		\begin{Prop}
		\label{thm:sens}
			Consider the perturbed version of the program \ref{eq:P_lp}
				\begin{equation}
					\left\{
					\begin{array}{lll}
						\text{minimize} & \sum\limits_{i \in I} c_i\tr x_i\\
						\text{subject to} & \sum\limits_{i \in I} H_i x_i \preceq b  + \eps \ind\\
						& x_i \in \conv(X_i) &  i \in I,
					\end{array}
					\right.
					\label{eq:P_lp_eps}
					\tag{$\mathcal{P}_{\mathrm{LP}}(\eps)$}
				\end{equation}
			whose optimal value is denoted by $\JPLPp{\eps}$. Let $\data{i} \Let (X_i, H_i,c_i)$ be the tuple representing the data of the $i^\text{th}$ subsystem, where the sets $X_i$ are all compact. Then, there exist a partition $I = I_1 \cup I_2$ and a constant $L(I_2) \Let L\big( (\data{i})_{i\in I_2} \big)$, only depending on the data of subsystems indexed by $I_2$, such that $|I_2| \le m$ and 
			\begin{align*}
				0\le \JPLPp{0} - \JPLPp{\eps} \le L(I_2) \eps, \qquad \forall \eps \in \R_+.
			\end{align*}
		\end{Prop}

\begin{proof}
	The proof, along with some preliminaries, is in Appendix \ref{sec:pf:sens}.  
\end{proof}

This result allows us to provide the following performance bound on the optimality gap for the recovered solutions.

\begin{Thm}[performance without Slater]
\label{thm:new_solutions_quality}
	Suppose the programs $\overline{\mathcal{P}}_{\mathrm{LP}}$ and $\overline{\mathcal{D}}$ satisfy Assumption \ref{assumption:uniqueness}. Then, any solution $\xsb \in \mxsb$ recovered satisfies
	\begin{eqnarray}
		J_{\mathcal{P}}(\xsb) - \JP \leq m \cdot \max_{i \in I} \gamma_i + \max_{\begin{subarray}{c}{I_2 \subset I} \\ |I_2| \le m \end{subarray}}L(I_2) \cdot \|\rho\|_\infty
	\label{eq:new_performance_bound}
	\end{eqnarray}
	where $\gamma_i$ and $\rho$ are as defined in \eqref{eq:duality_gap_bound} and \eqref{rho}, respectively, and $L(I_2)$ is the constant determined by subsystems indexed by $I_2$ as introduced in Proposition \ref{thm:sens}. 
\end{Thm}

The proof of Theorem \ref{thm:new_solutions_quality} essentially follows the same analysis of Section \ref{sec:proof_quality_of_solns}.
In light of this theorem, it is then clear that if $\{\gamma_i\}_{i\in I}$ and $\{L(I_2)\}_{I_2 \subset I}$ are uniformly bounded, and $\JP$ grows linearly with $|I|$, we reach the same conclusion on the optimality gap behavior as in \eqref{eq:vanishing_opt_gap}. These uniform bounds are satisfied if the diversity of the subsystems added to the problem, when we increase its size, is limited.


\section{Application Example: Charging of Plug-in Electric Vehicles (PEVs)}
\label{sec:4_example}

We consider a fleet of $|I|$ Plug-in (Hybrid) Electric Vehicles (PEVs) that must be charged by drawing power from the same electricity distribution network. As the number of PEVs increases, it becomes necessary to manage their charging pattern in order to avoid excessive stresses on the lines and transformers of the network. The role of interfacing the fleet of PEVs with the network operators is taken over by a so-called aggregator. 

In this Section we take the perspective of such an aggregator. Its control task is to assign charging slots to each individual PEV under its jurisdiction. The charging schedules have to be compatible with the local requirements (e.g., a desired final state of charge SoC), as well as global, network wide constraints.

 

\subsection{Model}

We will only consider the problem of establishing a feasible overnight charging schedule, since this is this period when most charging will occur \cite{OR_EVs}. We will also assume that at the time when the schedule is to be decided (e.g., midnight), all PEVs are connected and their local charging requirements (initial and final required SoC) have been communicated to the aggregator. Both of these assumptions can be easily relaxed by buffering newly connected PEVs, and recomputing every 20 minutes a charging schedule with the new population information, in a receding horizon fashion similar to \cite{ita_EVs}. Further, we assume that charging can be interrupted and resumed, but in order to avoid excessive switching, once charging starts it must continue for at least 20 minutes. This is a reasonable way of charging Lithion-Ion batteries, the most common in PEVs, because they do not present memory effect \cite{search_for_better_batteries}. Non-interruptible charging is not discussed here as it is uncommon in practice, but those applications for which it may be necessary (e.g., Nickel-Cadmium batteries) can be readily incorporated in our proposed framework with an appropriate design of the local constraints. We thus split the overnight period in intervals of 20 minutes each, and assume that the aggregator has authority to flag, for each individual PEV, the available charging time slots. 

For each PEV $i \in I$, charging at the time step $k$ is allowed when $u_i[k] = 1$, otherwise $u_i[k] = 0$. We will also consider as a separate case the situation in which discharging (or vehicle-to-grid V2G) is possible. Then, the discharge requests are modelled using $v_i[k] \in \zo$. Charging and discharging rates $P_i$ are assumed to be constant, as done in \cite{ita_EVs,callaway_hiskens,hiskens_kundu,ufuk_discrete_charging,my_EV_CDC_paper} and reflecting the charging station protocol IEC 61851\footnote{This is particularly true in case of stations with low power ratings. More generally, smart charging stations compatible with the IEC 61851 standard could operate in a semi-continuous fashion, i.e., with a minimum current output when charging, that can be then modulated in a certain band. This requirement results in disjunctive models of the corresponding subsystems, which requires discrete variables and thus fits our proposed framework. However we do not consider this aspect in the model.}.

The objective of the aggregator is to maximize the profit while satisfying the local charging requirements of each individual PEV and the network constraints, which are established by the network operator. The optimization problem model we work with is as follows.


\begin{itemize}
	\item \textbf{Subsystems model}. The subsystems controlled are the PEVs batteries. Battery's $i$ charge level is denoted by $e_i[k]$, its initial state of charge is $\einit$, which by the end of the charging period has to attain at least $\eref$. The charging conversion efficiency is $\zeta_i^u \eqdef 1- \zeta_i$, while the discharging efficiency is $\zeta_i^v \eqdef 1 + \zeta_i$\footnote{The discharging efficiency must be greater than 1. This correctly encodes the fact that the amount of energy fed back to the network is smaller than the battery's energy content decrease.}. We denote by $\emin$ and $\emax$ the battery's capacity limits. We thus have
	\begin{subequations} 
	\begin{align}
			& e_i[0] = \einit
			\label{eq:model_e_init}\\
			& e_i[k+1] = e_i[k] + P_i \Delta T \bigl(\zeta_i^u u_i[k] - \zeta_i ^v v_i[k]\bigr)
			\label{eq:EV_state_update}\\
			& e_i[N] \geq \eref
			\label{eq:EV_energ_requirement}\\
			& \emin \leq e_i[k] \leq \emax\\
			& u_i[k] + v_i[k] \leq 1
			\label{eq:model_no_ctrl_overlap}\\
			& u_i, v_i \in \zo^{N}.
	\label{eq:PEV_battery}
	\end{align}
	\end{subequations}
	Condition \eqref{eq:model_no_ctrl_overlap} removes the possibility of charging and discharging simultaneously. 
	\item \textbf{Coupling constraints}. Within a distribution system, network congestions typically occur on the lines departing from the substation, since the power flow at that point is the sum of all the power loads in the network, and thus largest \cite{lopes_integration_2011}. We therefore model congestion avoidance as a limit on the global aggregate charging and discharging power flow,
	\begin{align}
	 & \pmin[k] \leq \sum\limits_{i \in I} P_i(u_i[k] - v_i[k]) \leq \pmax[k].
	 \label{eq:PEV_coupling}
	\end{align}
	In cases when other network points are susceptible to congestions, similar coupling constraints have to be added, in which the sum is over a smaller subset of PEVs. Then Theorem \ref{thm:reducing_contraction_with_rank_hi} can be used to limit the necessary contraction.
	 
	 \item \textbf{Objective function}. The objective function encodes the cost the aggegator incurs to charge its fleet,
 	\begin{equation}
	 		\underset{u_u,v_i}{\mathrm{minimize}} \ 
	 		\sum\limits_{i \in I} \sum\limits_{k=0}^{N-1} P_i  \cdot \left(C^u[k] u_i[k] - C^v[k] v_i[k]\right) 
	 		\label{eq:PEV_objective}
 	\end{equation}
 	where $C^u$ and $C^v$ are, respectively, the price vector for electricity consumption and injection. We allow for time varying and possibly different charging and discharging prices. In the simulations we assume a 10\% markup on injection pricing, i.e., $C^v = 1.1 \cdot C^u$, which the system operator pays to the aggregator in order to incentivize PEVs to make the V2G functionality available. 
\end{itemize}
We can write the complete optimization program \eqref{eq:model_e_init}--\eqref{eq:PEV_objective} as
\begin{equation}
\left\{ 
\begin{array}{lll}
	\underset{e,u,v}{\text{minimize}} & \sum\limits_{i \in I} P_i \left(C^u \cdot u_i - C^v \cdot v_i\right) \\
	\text{subject to} 
	& P^{\text{min}} \leq \sum\limits_{i \in I} P_i(u_i - v_i) \leq P^{\text{max}}\\
	& (e_i, u_i, v_i) \in X_i
\end{array}
\right. 
\label{eq:EV_primal_problem}
\end{equation}
with
\begin{equation}
	X_i = \left\{ \left. \begin{bmatrix}
	e_i\\	u_i\\ 	v_i
	\end{bmatrix} \in \mathbb{R}^N \times \mathbb{Z}^{2N}  \right|
		\text{Eq. } \eqref{eq:model_e_init}-\eqref{eq:PEV_battery}
	\right\}.
\end{equation}

\begin{Remark}
\label{remark:old_result_not_ok_with_EV}
Note that the assumption in Theorem \ref{thm:bound_on_duality_gap} does not apply to this model. To see this, we consider the charge--only case. According to \eqref{eq:EV_primal_problem}, $H_ix_i = P_iu_i$, and a fractional $x_i \in \conv(X_i)$ implies that in at least one time step, charge is happening at a partial rate. To rectify it, one has to either increase it to the fixed charge rate or decrease it to 0. In the latter case it may however be necessary to increase charging at another time step, in order to satisfy the energy requirement of the EV \eqref{eq:EV_energ_requirement}. Since any such rectification will cause an increase of resources used at some time, the assumption cannot be met.
\end{Remark}

\subsection{Solution Method}
\label{sec:EV_soln_method}

We apply the method proposed in Theorem \ref{theorem_main_tight} to problem \eqref{eq:EV_primal_problem}, which we consider under two different scenarios: in the first, only charging is allowed ($v = 0$), while in the second, both charging and V2G controls are enabled. This allows us to illustrate how the method can be adapted in two cases in which the combinatorial structure of the subsystems is substantially different.

In both cases, the number of coupling constraints is $2N$. However, since these are box constraints, in consideration of Remark \ref{thm:reducing_contraction_with_rank_hi} we can reduce this number to $N$. Hence, the necessary contractions introduced in \eqref{rho}, for the charge only scenario and the case in which V2G is available, are, respectively,
\begin{equation}
\label{eq:EV_contractions}
\begin{array}{rl}
	\rho_{\mathrm{V2G}} & = N \cdot \max\limits_{i \in I} \left( \max P_i(u_i - v_i) - \min P_i(u_i - v_i) \right)\\
	 & = 2N \cdot \max\limits_{i \in I} P_i\\
	\rho_{\mathrm{charge}} & 
	= N \cdot \max\limits_{i \in I} P_i.\\
\end{array}
\end{equation} 
Dualizing the complicating constraints leads to the dual problem
\begin{equation}
\label{eq:EV_dual_problem}
\begin{array}{lll}
	\underset{\lambda, \mu}{\sup} & 
	\sum\limits_{i \in I} \underset{(e_i,u_i,v_i) \in X_i}{\text{min}} P_i \left(  ( C^u + \delta_i^u - \lambda + \mu)u_i - ( C^v + \delta_i^v - \lambda + \mu)v_i \right)\\
	& + \left( \lambda \pminb - \mu \pmaxb \right)\\
	\text{s.t.} & \lambda, \mu \succeq 0,
\end{array}	
\end{equation}
in which $\lambda[k]$ is the dual variable associated with the lower power rating constraint $\pminb[k] \eqdef \pmin[k] + \rho$, and $\mu[k]$ is the variable for $\pmaxb[k] \eqdef \pmax[k] - \rho$. We note that the cost vector for the subsystems is highly symmetric -- every PEV receives the same price profile. In order to ensure that Assumption \ref{assumption:uniqueness} is satisfied, we introduce small additive perturbation terms $\delta_i^u$ and $\delta_i^v$ to the costs $C^u$ and, respectively, $C^v$.

For the outer (maximization) problem in \eqref{eq:EV_dual_problem} we use a subgradient method \cite{anstreicher_two_well_known} with a constant stepsize rule, which we decrease every $20-30$ iterations. 

The inner (minimization) problem, on the other hand, is decomposed into $|I|$ decoupled subproblems
which are optimal control problems of 1-dimensional systems. For the sole charging case, the optimal local strategy can be proven to be greedy\footnote{Optimality of greedy can be shown using a Dynamic Programming argument, but since it is straightforward we omit it for brevity.}: the least number of charging steps is performed, and those are selected at times of ''lowest local prices'' (i.e., taking into account $\lambda$ and $\mu$ as well). The local optimizations are thus computationally inexpensive in this case. For the V2G case, on the other hand, the optimal charging and discharging strategy is not as immediate, so it must be solved either as a generic optimization problem, or by applying the Dynamic Programming (DP) algorithm, see e.g\  \cite[p.23]{bertsekas_DP}. In our tests we apply DP.

\subsection{Simulation Setup}

We compare the performance of our proposed method with the results provided by CPLEX 12.5. For each fleet size considered, we generate 10 random instances based on the parameters provided in Table \ref{tab:EV_data} in Appendix \ref{sec:7_sim_results}. In order to ensure a fair comparison, since CPLEX is generally unable to find exact solutions to the model \eqref{eq:EV_primal_problem}, we first run our proposed algorithm on each problem instance, record the optimality gap (we get a tight lower bound for free as a by-product of our method), and then run CPLEX up to the same optimality gap. Furthermore, the perturbation $\delta_i$ is added to the objective function, and the perturbed problem is the one on which we deploy both our method as well CPLEX. This ensures that both methods are exposed to exactly the same problem.
All our tests are performed on a Desktop PC with 8GB of RAM and a 3.10 GHz processor.

\subsection{Results}
\label{sec:EV_results_exp}

Figure \ref{fig:optimality_gap} illustrates the optimality gap of the recovered solutions (min, max and average). The asymptotic behaviour \eqref{eq:vanishing_opt_gap} is confirmed. 

\begin{figure}
\centering
	\includegraphics[width=0.70\linewidth]{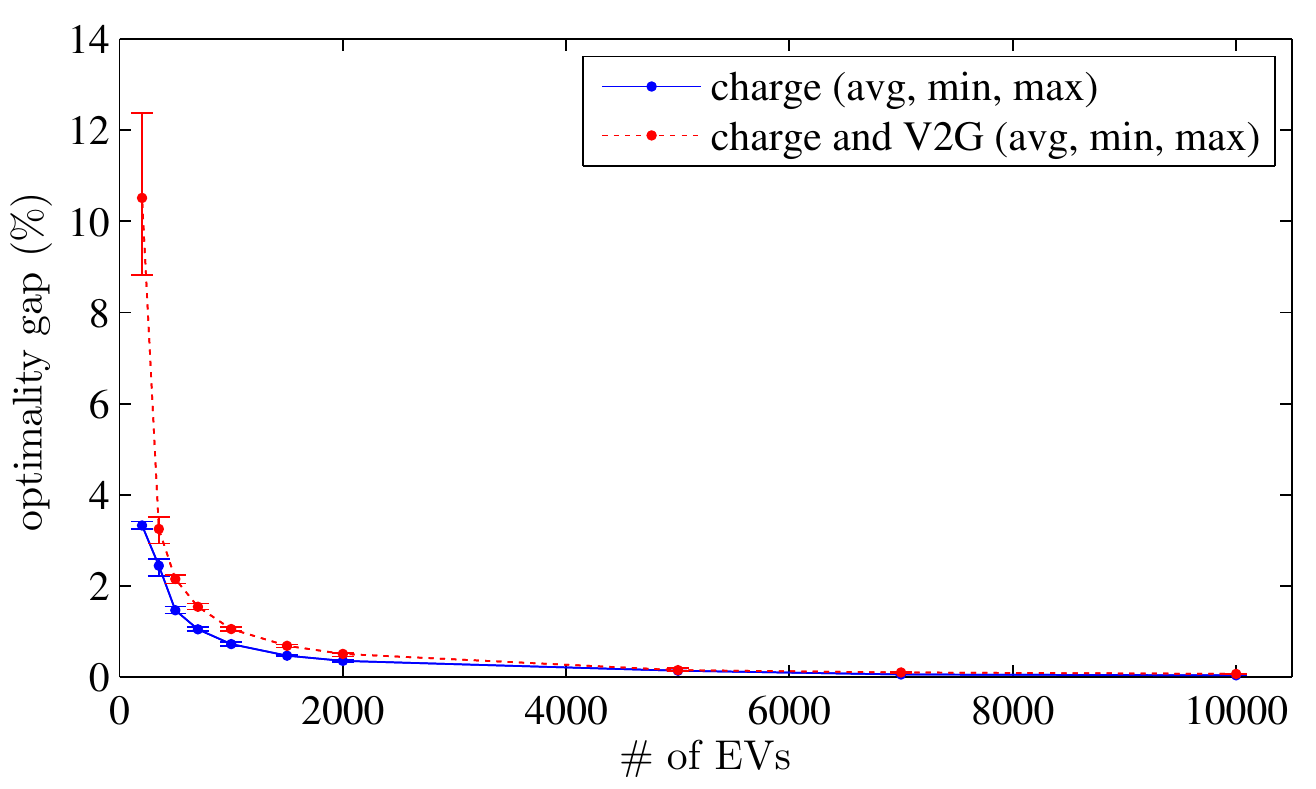}
	\caption{Optimality gap of the solutions recovered using the proposed method.}
	\label{fig:optimality_gap}
\end{figure}

Solution times are shown on Figure \ref{fig:optimality_gap_and_solve_time}. Owing to the greedy subproblem structure when only charging is allowed (discussed in the previous Section \ref{sec:EV_soln_method}), computation times in this case are fast: using our method, the largest instances are consistently solved within 5 seconds, see Figure \ref{fig:solve_time}. CPLEX is comparably fast. Figure \ref{fig:solve_time_v2g} shows solve times when the discharging functionality is enabled. V2G introduces a much more complicated combinatorial subproblem structure -- the optimal local control is not greedy anymore. In this case solution via CPLEX is impractical, because solve times are generally long and affected by substantial variances. For the case with 500 PEVs, solution times vary from 15 minutes to 4 and a half hours, and up to 6 hours on the two instances that CPLEX wasn't able to solve before running out of memory. Our proposed method has the advantage of providing consistent solution times across different instances, and the solution times substantially outperfom CPLEX also on those instances in which CPLEX provides a solution at all. It should be emphasized that the computations are carried out on a single processor, so that solve times can be reduced substantially by exploiting parallelism.
 

\begin{figure}
\centering
\subfigure[Solve time (charge only)]{
	\includegraphics[width=0.47\linewidth]{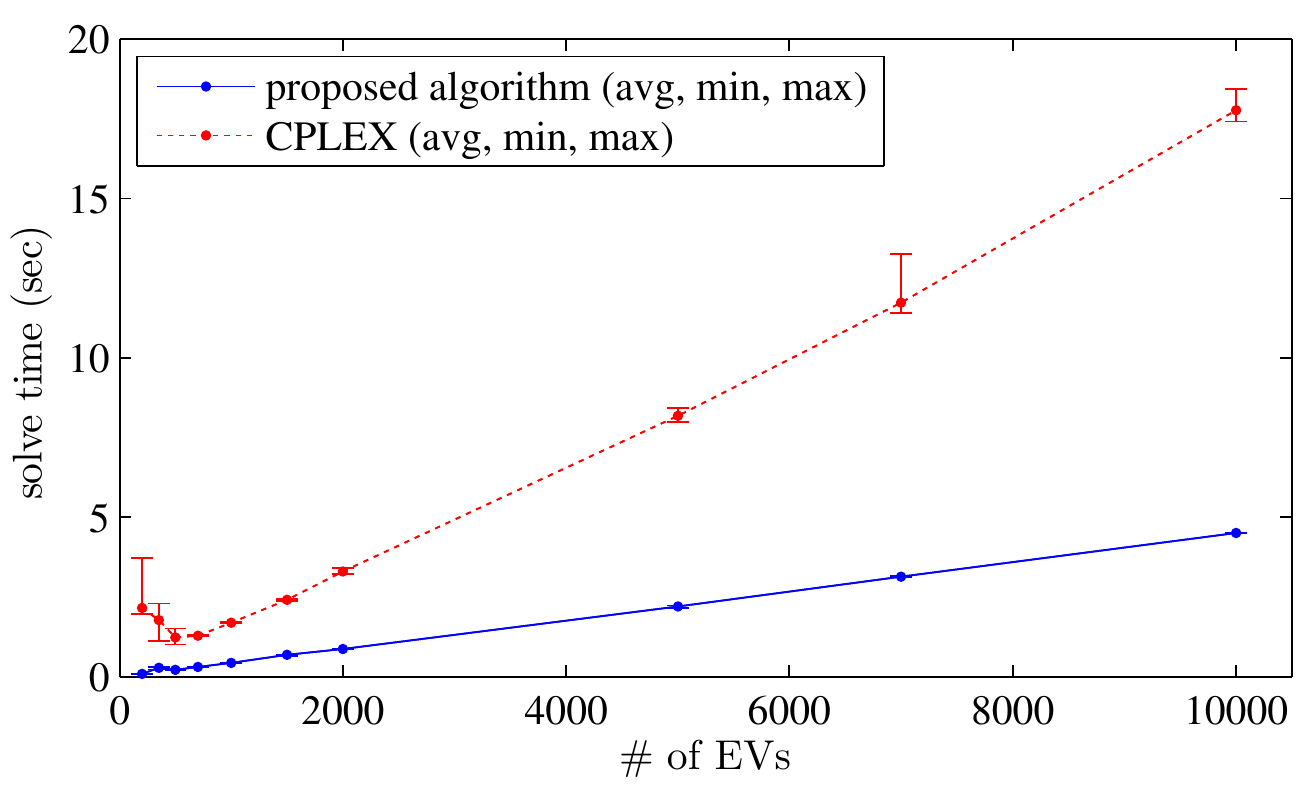}
	\label{fig:solve_time}
}
\vspace{-2mm}
\subfigure[Solve time (charge and V2G)]{
	\includegraphics[width=0.47\linewidth]{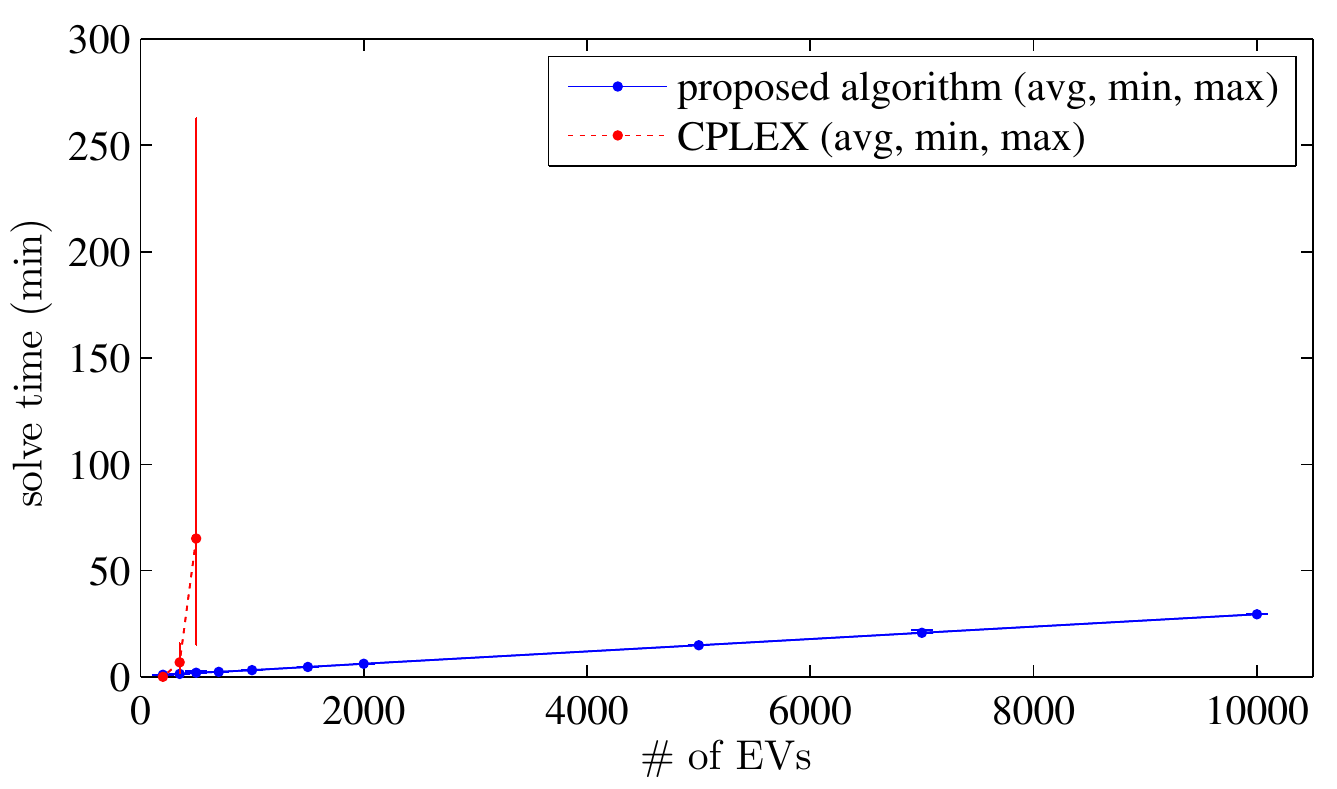}
	\label{fig:solve_time_v2g}
}
\caption{Solve times.}
\label{fig:optimality_gap_and_solve_time}
\end{figure}

\begin{figure}
\label{fig:dual_objective}
\centering
\subfigure[Dual objective]{
	\includegraphics[width=0.46\textwidth]{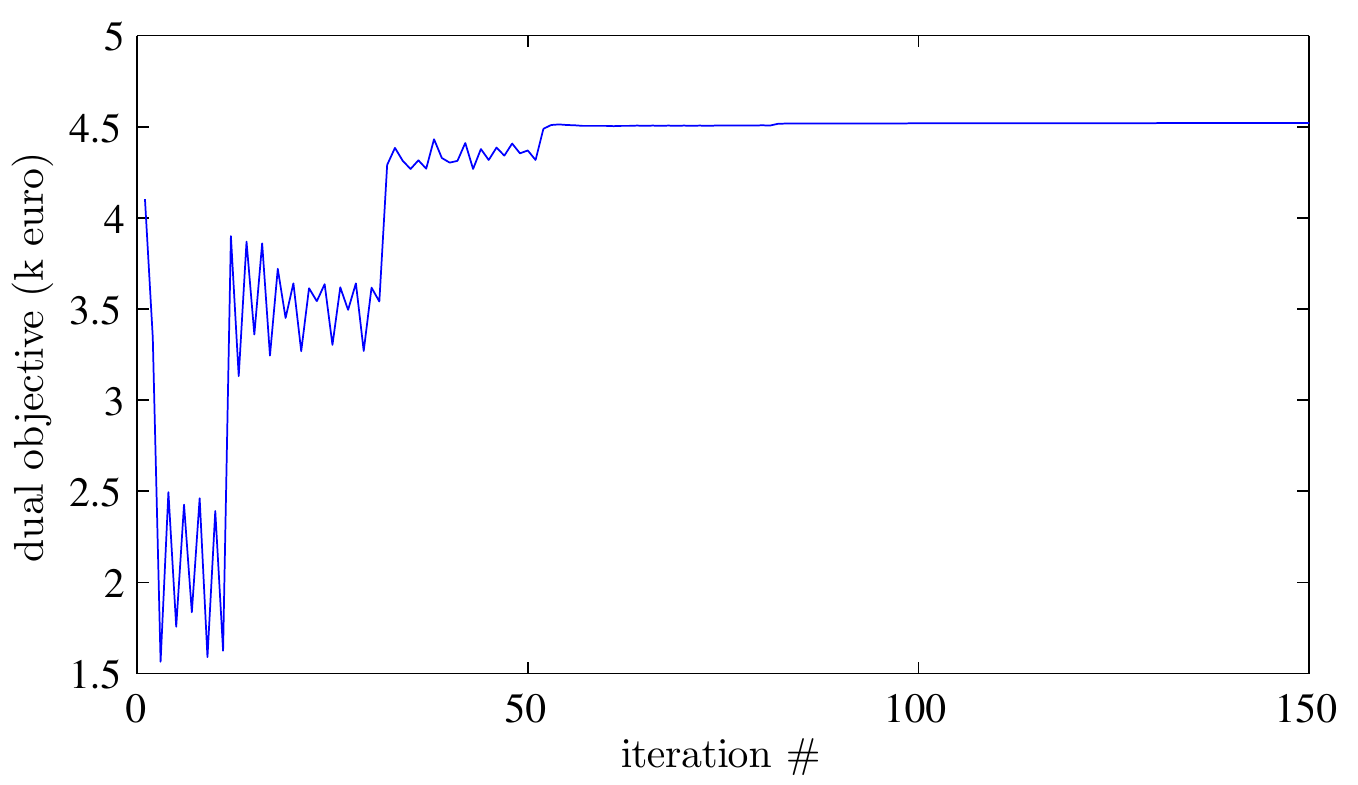}
	\label{fig:dual_objective_1}
}
\quad
\subfigure[Infeasibility]{
	\includegraphics[width=0.46\textwidth]{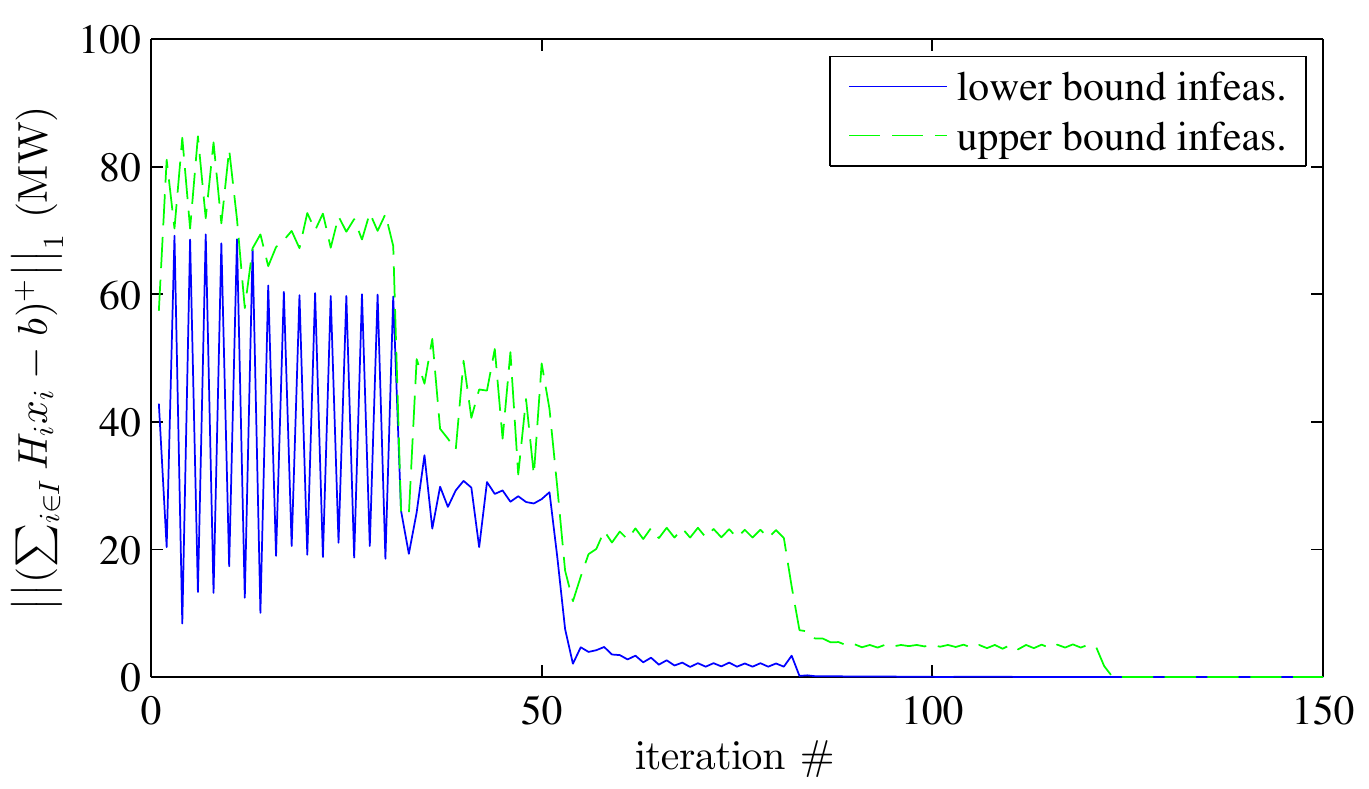}
	\label{fig:dual_objective_2}
}
\caption{Dual function value and feasibility violations at each iteration.}
\end{figure}

Figure \ref{fig:dual_objective_1} and \ref{fig:dual_objective_2} show the typical convergence behavior for the dual objective and the coupling constraints violations. 
Note that inner solutions are feasible starting from iteration $\sim 120$, while one may have interrupted the dual method already at iteration $\sim 60$ given the dual objective behaviour. 

Finally, Figure \ref{fig:plot_v2g_local_charge_profiles} depicts the local charging behaviour of one individual PEV. Charge and discharge control signals, as well as the evolution of the SoC are shown. The desired final state of charge is achieved by the end of the charging period.

The numeric values of these results are reported in the Appendix, see Table \ref{tab:EV_charge_result} for the charge-only experiments, and Table \ref{tab:EV_V2G_result} for the results with V2G.


\begin{figure}
\centering
	\includegraphics[width=0.65\linewidth]{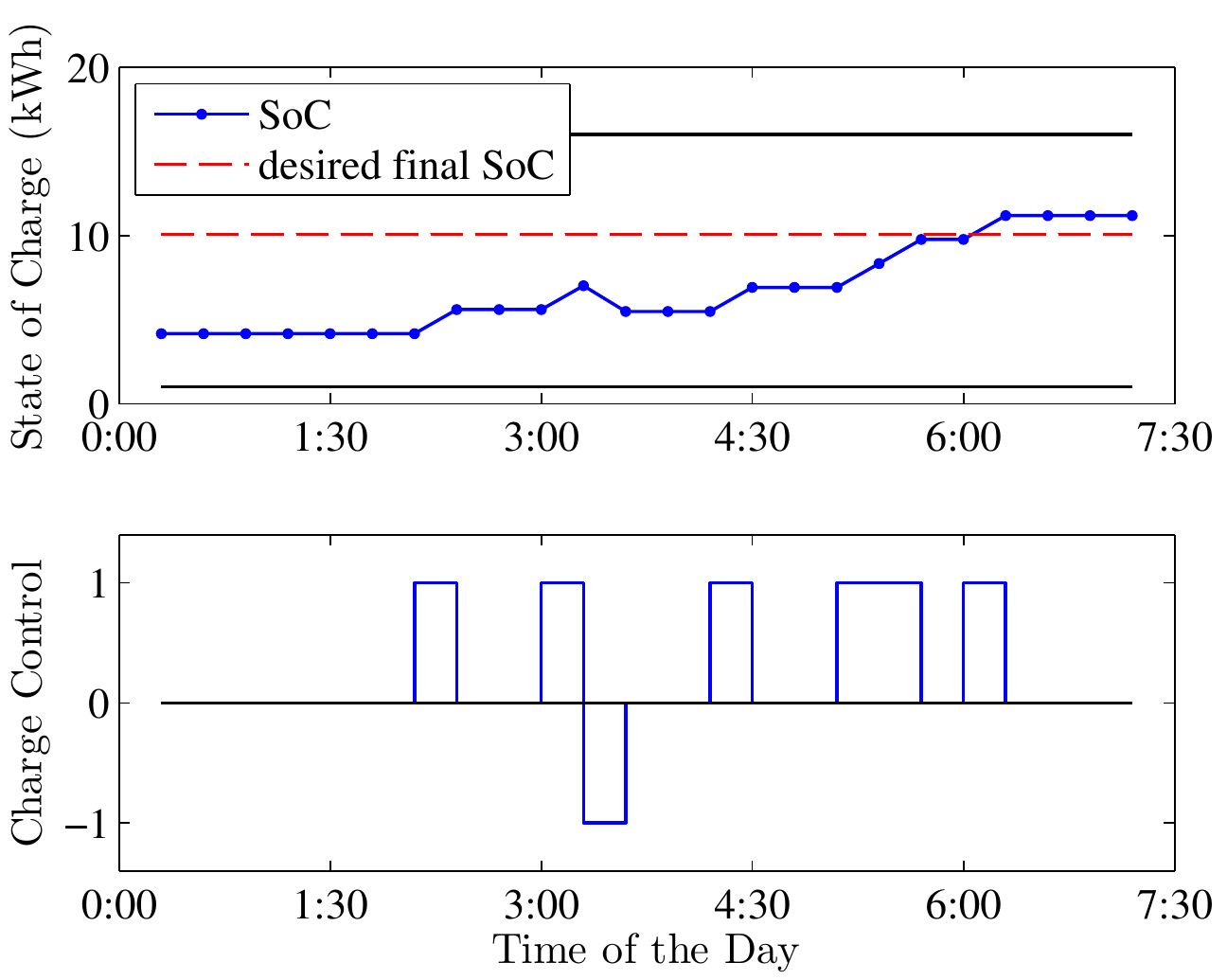}
	\caption{Local charge profile.}
	\label{fig:plot_v2g_local_charge_profiles}
\end{figure}

\section{Conclusion}
\label{sec:5_conclusion}
We have provided new results concerning the primal solutions recovered from lagrangian duals of problems structured as \ref{eq:P}. These results are of direct practical interest, in particular if one wishes to distribute the computational burden of calculating solutions to very large instances of such mixed integer programs. The strength of our results lies in the generality of $X_i$, which can include very sophisticated local models and therefore accommodate a large variety of practical applications.

It appears that many solution approaches can be derived from the result given in Theorem \ref{thm_1}; the one we propose in Section \ref{eq:P} is amenable to distributed computations and is simple to implement. It is also independent of the method used to solve the dual. Depending on the method chosen, convergence results could also be derived. One can for instance deploy the scheme exposed in \cite{anstreicher_two_well_known} together with our contraction method to recover an optimizer $\xslp$ of $\Plpbar$. According to Theorem \ref{thm_1}, this solution is known to satisfy integrality for at least $|I|-m$ subsystems. The non-integral components can be resolved by performing at most $m$ local optimizations, neglecting the coupling constraints. Owing to the contraction, the resulting solution retains feasibility, and satisfies performance bounds similar to \eqref{eq:performance_bound}.

%
\renewcommand{\thesection}{A}
\section{Appendix: Proofs}
\setcounter{equation}{0}
\numberwithin{equation}{section}
\label{sec:6_appendix}
\subsection{Proof of Fact \ref{prop:there_are_vertex_solns}} 
\label{sec:proof_fact_1}
\begin{proof} 
	Due to the linearity of the objective function and the definition of the set $X_i$, it is straightforward to observe that 
		\begin{align}
		\label{eq:min}
		\begin{array}{ll}
			\min\limits_{x_i \in X_i}(c_i\tr + \lambda\tr H_i)x_i & = \min\limits_{x_i \in \conv(X_i)}(c_i\tr + \lambda\tr H_i)x_i \\
			& =\min\limits_{x_i \in \vert(X_i)}(c_i\tr + \lambda\tr H_i)x_i.
		\end{array}
		\end{align}
	Thus, the desired assertion readily follows from the fact that $X_i$ are non-empty. 
\end{proof}

\subsection{Proof of Theorem \ref{thm_1}} \label{sec:proof_connectedness}
\begin{proof} 
	Let us introduce two new LPs that are crucial for our subsequent analysis. First, we denote by $\xij$ the $j$-th element of $\vert(X_i)$ for $j \in J_i$ where $J_i = \left\{1, \dots, |\vert(X_i)|\right\}$. In view of \eqref{eq:min}, one can derive an LP version of the program \ref{eq:D} as
		\begin{eqnarray*}
		\label{eq:hidden_lp_version_of_d}
			\left\{ 
			\begin{array}{ll}
				\underset{\lambda}{\text{maximize}} & - \lambda\tr b + \sum\limits_{i \in I} \min\limits_{j \in J_i}  \left( c_i\tr \xij  + \lambda \tr H_i \xij \right) \\
				\text{subject to} & \lambda \succeq 0,
			\end{array}
			\right.
		\end{eqnarray*}
	which can then be cast as the LP
		\begin{equation}
		\left\{ 
		\begin{array}{lll} 
			\underset{\lambda,z,s}{\text{maximize}} & - \lambda\tr b + \sum\limits_{i \in I} z_i\\
			\text{subject to} & z_i = c_i\tr \xij  + \lambda\tr H_i \xij - \sij &  i\in I, j \in J_i\\
			& \sij \geq 0 & i\in I, j \in J_i\\
			& \lambda \succeq 0,
		\end{array}
		\right. 
		\label{eq:D_lp}
		\tag{$\mathcal{D}_{lp}$}
		\end{equation}
	where $\sij$ is the slack variable, and $z_i$ corresponds to the inner problem {\small $\min\limits_{j \in J_i} ( c_i\tr \xij  + \lambda\tr H_i \xij)$}. The second LP is the dual program of \ref{eq:D_lp} described as
		\begin{equation}
		\left\{ 
		\begin{array}{lll}
			\underset{p}{\text{minimize}} & \sum\limits_{i \in I} \sum\limits_{j \in J_i} \pij c_i\tr \xij \\
			\text{subject to} & \sum\limits_{i \in I} \sum\limits_{j\in J_i} \pij H_i \xij \preceq b\\
			& \sum\limits_{j \in J_i}\pij = 1 &  i\in I\\
			& \pij \geq 0 &  i \in I, j \in J_i,
		\end{array}
		\right. 
		\label{eq:PP_lp}
		\tag{$\mathcal{P}_{lp}$}
		\end{equation}
	where $\pij \in [0,1]$ is the scalar optimization variable associated to the vertex $\xij$. Let us denote by $\ps$ an optimizer of \ref{eq:PP_lp}. Note that \ref{eq:PP_lp} corresponds to an extended LP version of \ref{eq:P_lp}, yet they are not entirely equivalent problems. In particular, each $\ps$ leads to a unique $\xslp$, but the reverse does not hold, i.e., uniqueness of $\xslp$ does not imply uniqueness of $\ps$. We split the proof of the theorem by proving the following steps:

	\begin{enumerate}[(a)]
		\item \label{step:1} Let $I_1 \subset I$ be a subset of indices where $\xislp \in \vert(X_i)$ for all $i \in I_1$. Then, $\xislp$ is an optimizer of the inner problem, i.e., $\xislp \in \argmin{x_i \in X_i}(c_i\tr x  + {\ls}\tr  H_i x)$ where $\ls$ is an optimizer of \ref{eq:D}. 
		
		\item \label{step:2} Let $(\ls,z^\star, \sss)$ be an optimal solution of \ref{eq:D_lp} and $\ps$ be an optimal solution of \ref{eq:PP_lp} with the corresponding optimizer $\xslp$ for \ref{eq:P}. If the optimal pair $(\ps, \sss)$ is strictly complementary, then $\xislp = \xis$ for all $i$ in the subset $I_1$ as defined in \eqref{step:1}. 
		
		\item \label{step:3} If $\xslp$ is a vertex for the program \ref{eq:P}, then the subset $I_1$ in \eqref{step:1} can be selected such that $|I_1| \ge |I| - m$. 
		
	\end{enumerate}
	
	Before proceeding with the proofs of the above results, let us highlight how the desired assertion, under the unique primal and dual optimizers, follows from these three steps. First, note that if the optimal solution of \ref{eq:D} is unique, then $(\ls,z^\star, \sss)$ is the unique solution to \ref{eq:D_lp}: $\ls$ coincides for \ref{eq:D} and \ref{eq:D_lp} according to \cite[p. 89]{geoffrion_74}; $z^\star$ is the optimal objective of the $i$-th inner problem, and is thus uniquely determined for fixed $\lambda$; and finally $\sijs$ is also uniquely determined by the equality constraints in \ref{eq:D_lp}, in which it is the only variable left undetermined. Therefore, $\sss$ always belongs to the pair $(\ps, \sss)$ of primal-dual optimizers for which strict complementarity holds; the existence of such pair is guaranteed in the LP setting \cite[Thm.\ 2.1]{greenberg}. Moreover, if $\xslp$ is unique, then it is always a vertex. Hence, the requirements of the above results are fulfilled and the theorem assertion is concluded. 
	
	\paragraph{Proof of \eqref{step:1}:} Let $\xislp \in \vert(X_i)$. Then, owing to the uniqueness of $\xslp$, for any solution $\ps$ of \ref{eq:PP_lp} we have $\pijsh = 1$ for the corresponding $\hat \jmath \in J_i$. Therefore, by complementary slackness, the dual optimizer has $\sijsh = 0$, and the step \eqref{step:1} follows by
		\begin{align}
		\label{eq:z}
			z_i^\star = c_i\tr \xislp + {\ls}\tr H_i \xislp \le c_i\tr \xij  + {\ls}\tr  H_i \xij, \qquad \forall j \in J_i.
		\end{align}

	\paragraph{Proof of \eqref{step:2}:} Let $i\in I_1$ and, as explained in the proof of \eqref{step:1}, $\pijsh = 1$ for the corresponding $\hat \jmath \in J_i$. In light of the equality constraint $\sum_{j\in J_i} {\pij} = 1$, we have $\pijs = 0$ for all $j \neq \hat \jmath$. The assumed strict complementarity now  implies $\sijs \neq 0$ for all $j \neq \hat \jmath$, which leads to a strict inequality in \eqref{eq:z}. Hence, the inner problem $\min_{x_i \in X_i}(c_i\tr x  + {\ls}\tr  H_i x)$ has the unique solution $\xis$. Now the desired assertion follows from the step \eqref{step:1}.

	


	\paragraph{Proof of \eqref{step:3}:} Problem \ref{eq:PP_lp} has $m$ inequality constraints ($b \in \mathbb{R}^m$) and $|I|$ equality constraints, plus the positivity constraints on $\pij$. We can add slack variables to the complicating constraints thus obtaining a problem with $|I|+m$ equality constraints and positivity constraints on all the optimization variables, which are now the slacks $q \in \mathbb{R}^m_+$ and the variables $\pij$. The constraints of \ref{eq:PP_lp} can therefore be rewritten as $\bigH (p\tr,q\tr)\tr = (b\tr,1\dots,1)\tr$, $p,q \geq 0$,
	where the matrix $\bigH$, is defined as
	\begin{align}
	\begin{small}
		\label{big matrix}
		\bigH = \left[ 
			\begin{array}{c:c:c|c|c:c:c|c}
				H_1 x_1^{1} & \cdots & H_1 x_1^{J_1} 
				&\cdots
				& H_{|I|} x_{|I|}^{1} & \cdots & H_{|I|} x_{|I|}^{J_{|I|}}
				& I_{m \times m}
				\\ \hline
				1 & \cdots & 1 
				&\cdots
				& 0 & \cdots & 0 
				& 0
				\\ \hline
				\enskip&\vdots&\enskip
				&\ddots
				&\enskip&\vdots&\enskip
				& \vdots
				\\ \hline
				\undermat{\bigH_1}{0 & \cdots & 0}
				&\cdots&
				\undermatwo{\bigH_{|I|}}{1 & \cdots & 1}
				& 0
			\end{array}
		\right]	
	\end{small}
	\end{align}
	\newline
	\newline
	in which we have also defined the submatrices $\bigH_i$, $i \in I$.
	It is well known (see \cite[Prop. 2.1.4 (b)]{bertsekas_cvx_theory}) that for a problem in this form any feasible point is a vertex if and only if the columns of $\mathbb{H}$ corresponding to the non-zero coordinates of the point are linearly independent. This is then true for any optimal vertex. Thus, $\mathrm{supp}(\ps) \leq |I|+m$, as the number of rows of $\bigH$ is $|I|+m$. On the other hand, the constraint $\sum_{j \in J_i} \pij = 1, \ i \in I$ in \ref{eq:PP_lp} forces any feasible solution to have at least one variable $\pij$ larger than zero for each $i \in I$, i.e.\ $\mathrm{supp}(\ps) \geq |I|$. It thus follows that at least $|I|-m$ entries must be set to 1 at any feasible vertex solution, including an optimal one.

\end{proof}

\subsection{Proof of Theorem \ref{theorem_main_tight}}
\label{sec:proof_contraction_works}
\begin{proof} 
	Note that by construction $\xsb \in \mxsb$ for all $i \in I$. Then, it only suffices to show $\sum_{i \in I} H_i \xisb \preceq b$. By virtue of Theorem \ref{thm_1}, we know that there exists a subset $I_1 \subset I$ such that $|I_1| \ge |I|-m$ and  $\xisb = \xislpb$. Setting $I_2 = I \setminus I_1$, we have
		\begin{align*}
			\sum\limits_{i \in I} H_i \xisb & = \sum\limits_{i \in I_1} H_i \xisb + \sum\limits_{i \in I_2} H_i \xisb \\
			& = \sum\limits_{i \in I_1} H_i \xislpb + \sum\limits_{i \in I_2} H_i \xisb \\
			& = \underbrace{\sum\limits_{i \in I} H_i \xislpb}_{\preceq \bar b} + \underbrace{\sum\limits_{i \in I_2} \big(H_i \xisb - H_i \xislpb \big)}_{\preceq \rho} \preceq b .
		\end{align*}
\end{proof}

\subsection{Proof of Theorem \ref{thm:solutions_quality}} \label{sec:proof_quality_of_solns}

\begin{proof}
	Note that 
	\begin{eqnarray*}
		J_{\mathcal{P}}(\xsb) - \JP & = 
		& \underbrace{ \left[ J_{\mathcal{P}}(\xsb) - J^\star_{\mathcal{\overline{P}}_{\mathrm{LP}}} \right] }_{\text{(i)}}  + 
		\underbrace{ \left[  J^\star_{\mathcal{\overline{P}}_{\mathrm{LP}}} - \JPLP \right] }_{\text{(ii)}}  \\
		& & + \underbrace{ \left[ \JPLP  - \JP \right] }_{\text{(iii)}},
	\end{eqnarray*}
	where each term can be bounded as follows:
\begin{enumerate}[(i)]
	\item According to Theorem \ref{thm_1}, there exists an index set $I_1$ with $|I_1| \geq |I|-m$ such that, for all $i \in I_1$, $\xislpb = \xisb$. Defining $I_2 \doteq I \setminus I_1$, we have 
	\begin{align*}
		J_\mathcal{P}(\xsb) - J_{\mathcal{P}}(\xslpb) & = \sum_{i \in I_2} \left( c_i\tr \xisb - c_i\tr \xislpb \right)\\
		& \leq m \cdot \max_{i \in I} \Big( \max_{x_i \in X_i} c_i\tr x_i - \min_{x_i \in X_i} c_i\tr x_i \Big) \\
		& = m \max_{i \in I} \gamma_i.
	\end{align*}
	
	\item By virtue of \cite[Lemma 1]{nedic_ozdaglar}, given the Slater's point $\wt x$ we can bound $\|\lsb\|_1$ by
		\begin{align*}
			\|\lambda^\star \|_1 \le \frac{1}{\zeta |I|}\bigg( \sum_{i\in I} c_i\tr \wt x_i  - \Big( \sum_{i\in I} \min_{x_i \in {X_i}}\left(c_i\tr + \lambda H_i\right) {x}_i \Big) - \lambda\tr  b \bigg),\\
			\quad \forall \lambda \succeq 0.
		\end{align*}
	Setting $\lambda = 0$ in the above, we arrive at
	\begin{align*}
			\|\lambda^\star\|_1 \le  \frac{1}{\zeta} \max_{i \in I} \gamma_i, \qquad \gamma_i \Let \max_{x_i \in \mathrm{X_i}} c_i\tr x_i - \min_{x_i \in \mathrm{X_i}} c_i\tr x_i.
	\end{align*} 
	In light of perturbation theory \cite[Sec.\ 5.6.2]{boyd_convex_book}, one can bound the term (ii) from above by $(\lsb)\tr  \rho$, where $\lsb$ is the optimizer of the program $\overline{\mathcal{D}}$ and $\rho$ is the contraction vector as defined in \eqref{rho}. Thus, 
	\begin{align*}
		J^\star_{\Plpbar} - \JPLP & \le (\lsb)\tr \rho \le \|\lsb\|_{1} \|\rho\|_\infty \le \frac{\|\rho\|_\infty }{\zeta} \max_{i \in I}  \gamma_i.
	\end{align*}
	\item By definition, $\overline{\mathcal{P}}_{\mathrm{LP}}$ is a relaxed version of \ref{eq:Pbar}. Hence $\JPLP  - \JP \leq 0$.
\end{enumerate}
\end{proof}

\subsection{Proof of Theorem \ref{thm:reducing_contraction_with_rank_hi}}
\label{sec:proof_reduced_conservatism}

For a given $\xs \in \mxs$, let us introduce $\mis{\I} = \left\{ i \in I \left| \right. \xislp \neq \xis \right\}$. 
For the $k$-th complicating constraint we then have
\begin{eqnarray*}
	\sum_{i \in I} H_i^k \xis & = & \sum_{i \in I \setminus \mis{\I}} H_i^k \xislp + \sum_{i \in \mis{\I}} H_i^k \xis\\
	& \leq & b + \sum_{i \in \mis{\I}} H_i^k (\xis - \xislp)\\
	& = & b + \sum_{i \in \mis{\I} \cap I_k} H_i^k (\xis - \xislp)\\
	& \leq & b + |\mis{\I} \cap I_k| \cdot \max_{i \in I_k} \left( \max_{x_i \in X_i} H_i^k x_i - \min_{x_i \in X_i} H_i^k x_i \right)
\end{eqnarray*}
In order to get a bound on $|\mis{\I} \cap I_k|$, we resort again to the program \ref{eq:PP_lp}. We know that, under Assumption \ref{assumption:uniqueness}, $\xis \neq \xislp$ if and only if $\xislp \notin \vert(X_i)$, as shown in Appendix \ref{sec:proof_connectedness}. Thus, if $i \in \mis{\I}$ there are at least two $j \in J_i$ such that $\piijs > 0$ in the corresponding program \ref{eq:PP_lp}. And for every $i \in I$, there is always at least one $j \in J_i$ such that $\piijs > 0$. Thus
\begin{align*}
	|\mathrm{supp}(\piiks)| \geq |\I_k \setminus \mis{\I}| + 2|\mis{\I} \cap \I_k| = |\I_k| + |\mis{\I} \cap \I_k|.
\end{align*}
On the other hand, in view of \cite[Prop.\ 2.1.4 (b)]{bertsekas_cvx_theory}, and as discussed in Appendix \ref{sec:proof_connectedness}, the columns within the matrix $\bigH$ (defined in Equation \eqref{big matrix}) corresponding to non-zero $\pijs$ coordinates must be linearly independent. Hence $|\mathrm{supp}(\ps)| \leq \rank(\bigH)$ and in particular
\begin{align*}
	|\mathrm{supp}(\piiks)| \leq \rank([\bigH_i]_{i \in \I_k}).
\end{align*}
Finally, from the structure of $\bigH$ defined in Equation \eqref{big matrix}, it is clear that
\begin{align*}
	\rank([\bigH_i]_{i \in \I_k}) \leq \rank([H_i]_{i \in \I_k}) + |\I_k|.
\end{align*}
Combining the above inequalities immediately leads to
\begin{align*}
	\rank([H_i]_{i \in \I_k}) \geq |\mis{\I} \cap \I_k|,
\end{align*}
as desired.

\subsection{Proof of Proposition \ref{thm:sens}} \label{sec:pf:sens}
	The objective is to establish a connection from the sensitivity of the large scale, but structured, optimization program \ref{eq:P_lp} to a reduced version in which only $m$ subsystems appear. To this end, we first start with some preparatory lemmas. 
	
	\begin{Lem}
		\label{lem:convex}
		Let $J:\R_+ \ra \R$ be a convex function. Suppose there exist a constant $L$ and a sequence $\left\{\eps_n\right\}_{n\in\N}$ such that $\eps_n \ra 0$ as $n$ goes to infinity and $J(0) - J(\eps_n) \le L\eps_n$ for all $n 
		\in \N$. Then, $J(0) - J(\eps) \le L \eps$ for all $\eps \in \R_+$.
	\end{Lem}
	
	\begin{proof}
		For the sake of contradiction, suppose there exists an $\bar \eps$ such that $J(0) - J(\bar \eps) > L \bar \eps$. Let $n$ be large enough so that $\eps_n \in (0,\bar \eps)$ and $\alpha \Let \frac{\eps_n}{\bar \eps}$. In light of convexity of $J$, we have 
			\begin{align*}
				J(\eps_n) & \le (1 - \alpha) J(0) + \alpha J(\bar \eps)\\ 
				& < (1-\alpha) J(0) + \alpha (J(0) - L \bar \eps) = J(0) - L \eps_n,
			\end{align*}
		which is obviously in contradiction with our assumption. 
	\end{proof}

	\begin{Lem}
		\label{lem:sens}
			Consider the parametrized LP
				\begin{align}
					\left\{
					\begin{array}{ll}\vspace{1mm}
						\underset{x}{\text{minimize}} & cx \\
						\text{subject to} & Ax \preceq b + \eps \ind,
					\end{array}
					\right.
					\label{eq:P_eps}
				\end{align}
			where $\eps \in \R_+$ is the parameter and $\ind \Let [1, \dots, 1]\tr \in \R^m$. Suppose the program admits a vertex optimizer whose objective value is denoted by $J(\eps)$. Then, there exists a constant independent of the resource vector $b$, denoted by $L(A,c)$, such that 
			\begin{align*}
				0 \le J(0) - J(\eps) \le L(A,c) \eps, \qquad \forall \eps \in \R_+.
			\end{align*}
	\end{Lem}
	
	\begin{proof}
		We only need to prove the right-hand side of the inequality as the left-hand side trivially holds since the parameter $\eps$ is non-negative and only relaxes the constraint. Let $x^\star(\eps)$ be a vertex optimizer for \eqref{eq:P_eps}.
		By virtue of \cite[Prop.\ 2.1.4 (a)]{bertsekas_cvx_theory}, given a fixed $\eps$, we know that there exists a collection of $m$ linearly independent rows of the matrix $A$, denoted by the invertible submatrix $\sub A(\eps)$, such that $\sub A(\eps) x^\star(\eps) = b + \eps \ind$. Note that the number of submatrices of matrix A is, of course, finite. Therefore, one can always pick a sequence $\left\{\eps_n\right\}_{n \in \N}$ such that $\eps_n \ra 0$ as $n$ goes to infinity and the corresponding submatrix $\sub A(\eps_n)$ is constant; let us denote this submatrix by $\sub A$. We thus have
			\begin{align*}
				J(0) - J(\eps_n) = c x^\star(0) - c x^\star(\eps_n) =  - c \sub{A}^{-1} \ind \eps_n \le L(c,A) \eps_n, 
			\end{align*}
		where the constant can be, for example, $L(c,A) \Let m \|c\|_2 \|\sub{A}^{-1} \|_2$. Note that, by construction, the submatrix $\sub{A}$ is invertible and the norm $\|\sub{A}^{-1} \|$ is bounded. The desired assertion now follows from the convexity of the perturbation mapping $\eps \mapsto J{(\eps)}$ \cite[Sec.\ 5.6.2]{boyd_convex_book} and Lemma \ref{lem:convex}.
	\end{proof}

	\begin{proof}[Theorem \ref{thm:sens}]
	\label{sec:proof_sens}
		Given the partition $I = I_1 \cup I_2$, we introduce a reduced version of \ref{eq:P_lp_eps} associated with the index set $I_2$ as follows:
			\begin{align}
				\left\{
				\begin{array}{lll}\vspace{1mm}
					\underset{(x_i)_{i \in I_2}}{\text{minimize}} & \sum\limits_{i \in I_2} c_i\tr x_i\\
					\text{subject to} & \sum\limits_{i \in I_2} H_i x_i \preceq b - \sum\limits_{i \in I_1} H_i \xislp + \eps \ind \\
					& x_i \in \conv(X_i) &  i \in I_2,
				\end{array}
				\right. 
				\tag{$\mathcal{R}_{I_2}(\eps)$}
				\label{eq:R_eps}
			\end{align}
		where $\xslp$ is an optimizer of the program \ref{eq:P_lp}. We denote the optimal value of \ref{eq:R_eps} by $\JRp{\eps}$. Let us highlight that for any partition of the index set $I = I_1 \cup I_2$ the program \ref{eq:R_eps} is always feasible as $\left(\xslp \right)_{i \in I_2}$ trivially satisfies the constraints for any $\eps \in \R_+$. As a first step in the proof, we show that there exist an index subset $I_2$ and a sequence of $\left\{\eps_n\right\}_{n \in \N}$ such that $|I_2| \le m$ and the optimal values $\JPLPp{\eps_n}$ and $\JRp{\eps_n}$ have the same sensitivity in terms of the parameter $\eps$. 
		
		
		Let $\xslp(\eps)$ be a vertex optimizer of the program $\Plp(\eps)$; the existence of such a vertex is always ensured since the feasible set of $\Plp(\eps)$ is a compact polytope. In light of part \eqref{step:3} in the proof of Theorem \ref{thm_1}, we know that for each $\xslp(\eps)$ there exists a partition $I = I_1(\eps) \cup I_2(\eps)$ where $|I_2(\eps)| \le m$ and $\xilp{\eps} \in \vert(X_i)$ for all $i \in I_1(\eps)$. Due to the fact that the number of the subsets of $I$ as well as the set $\vert(X_i)$ is finite, then there exists a partition $I = I_1 \cup I_2$ and a subsequence of $\left\{ \eps_{n}\right\}_{n \in \N}$ such that $|I_2| \le m$ and $\xilp{\eps_n}$ are constants for $i \in I_1$. By compactness we can, without loss of generality, assume that this sequence is convergent. It is a well-known result in the context of perturbation theory that the mapping $\eps \mapsto \JPLP(\eps)$ is convex on $[0,\infty)$, and in particular continuous \cite[Sec.\ 28]{rockafellar}. Hence, one can infer that $\xilp{\eps_n}$ converges to an optimizer of \ref{eq:P_lp}, which consequently implies $\xilp{\eps_n} = \xislp$ for all $i \in I_1$. Therefore, by construction of the auxiliary program \ref{eq:R_eps} we can deduce
			\begin{align*}
				\JPLPp{0} - \JPLPp{\eps_n} = \JRp{0} - \JRp{\eps_n}, \qquad \forall n \in \N.
			\end{align*}
		Now, in view of Lemma \ref{lem:sens}, we know that the right-hand side of the above equality is non-negative and can be upper bounded by a constant only depending on the data of the subsystems indexed in $I_2$, i.e., $(\data{i})_{i\in I_2}$. Let us denote this constant by $L(I_2)$. Then, we have 
			\begin{align*}
				0\le \JPLPp{0} - \JPLPp{\eps_n} \le L(I_2) \eps_n, \qquad \forall n \in \N,
			\end{align*}
		that by virtue of Lemma \ref{lem:convex} leads to the desired assertion.
	\end{proof}

\renewcommand{\thesection}{B}
\section{Simulation Tables}
\label{sec:7_sim_results}
Table \ref{tab:EV_data} contains the parameters used in the simulation. Values in brackets are sampled from a uniform distribution over the given interval. Tables \ref{tab:EV_charge_result} and \ref{tab:EV_V2G_result} report the numeric values of the performance results derived from the simulations discussed in Section \ref{sec:4_example}.

\vspace{1cm}

\begin{table*}[h!]
\begin{center}
\scalebox{0.9}{
\begin{tabular}{|l | c | c | c | c |c | c |}
\hline
Parameter & $|I|$ & $P_i$ & $\emin$ & $\emax$ & $\einit$ & $\eref$ \\
\hline
\hline
Unit & PEVs & kW & kWh & kWh & kWh & kWh \\
\hline
Value & $200-10000$ & $[3 ; 5]$ & $1$ & $[8;16]$ & $[0.2 ; 0.5] \cdot \emax$ & $[0.55 ; 0.8] \cdot \emax$\\ 
\hline
\end{tabular}
}

\begin{tabular}{c}
\vspace{-0.1cm}
\end{tabular}

\scalebox{0.98}{
\begin{tabular}{|l | c | c | c | c | c | c | c| c|}
\hline
Parameter & $\zeta_i$ & $\Delta T$ & $N$ & $\pmax$ & $\pmin$ & $C^u[k]$ & $C^v[k]$ &  $\delta_i^u, \delta_i^v$	\\
\hline
\hline
Unit & $-$ & min & $-$ & kW & kW & \euro/MWh & \euro/MWh & \euro/MWh\\
\hline
Value & $[0.015 ; 0.075]$ & $20$ & $24$ &
$3\cdot|I|$
& $- \pmax$ & $[19 ; 35]$	& $1.1 \cdot C^u[k]$ & $[-0.3 ; 0.3]$\\
\hline
\end{tabular}
}
\end{center}
\caption{Parameters used in the simulations. Values in the brackets are sampled from a uniform distribution.}
\label{tab:EV_data}
\end{table*}

\begin{table*}
\begin{center}
\begin{tabular}{l ccc c ccc c ccc}
\toprule
& \multicolumn{7}{c}{Proposed Method}& & \multicolumn{3}{c}{CPLEX}\\
& \multicolumn{3}{c}{Opt. Gap (\%)} & & \multicolumn{3}{c}{Solve time$^\dagger$ (sec)} & & \multicolumn{3}{c}{Solve time (sec)}\\
\cmidrule(l){2-4}
\cmidrule(l){6-8}
\cmidrule(l){10-12}
\# PEVs & Min & Avg 	& Max 	&& Min 		& Avg 	& Max 	&& Min 		& Avg 		& Max \\
\midrule
\midrule
200 	& 3.24	& 3.32	& 3.41	&& *		& *			& *			&& 1.97		& 2.16		& 3.74  \\
350		& 2.21	& 2.44	& 2.58	&& *		& *			& *			&& 1.13		& 1.79 		& 2.30	\\
500 	& 1.40	& 1.46	& 1.54	&& *		& *			& *			&& 1.02		& 1.24		& 1.52	\\
700 	& 1.01	& 1.05	& 1.10	&& 0.31		& 0.31		& 0.31		&& 1.27		& 1.29		& 1.31	\\
1000 	& 0.68	& 0.72	& 0.76	&& 0.44		& 0.44		& 0.44		&& 1.68		& 1.70		& 1.73	\\
1500 	& 0.46	& 0.47	& 0.49	&& 0.67		& 0.70		& 0.70		&& 2.39		& 2.42		& 2.45	\\
2000 	& 0.33	& 0.35	& 0.36	&& 0.88		& 0.88		& 0.89		&& 3.22		& 3.30		& 3.41	\\
5000 	& 0.13	& 0.14	& 0.14	&& 2.17		& 2.21		& 2.23		&& 8.00		& 8.18		& 8.43	\\
7000 	& 0.05	& 0.05	& 0.06	&& 3.14		& 3.15		& 3.16		&& 11.40		& 11.74		& 13.25	\\
10000 	& 0.03	& 0.03	& 0.04	&& 4.45		& 4.51		& 4.52		&& 17.41		& 17.77		& 18.43	\\
\bottomrule
\multicolumn{12}{l}{\small{(*) $\leq 0.3$ sec (imprecise measurements).}}\\
\end{tabular}
\end{center}
\caption{Charging only.}
\label{tab:EV_charge_result}
\end{table*}

\begin{table*}
\begin{center}
\begin{tabular}{l ccc ccc ccc}
\toprule
& \multicolumn{6}{c}{Proposed Method} & \multicolumn{3}{c}{CPLEX}\\
& \multicolumn{3}{c}{Opt. Gap (\%)} & \multicolumn{3}{c}{Solve time (min)} & \multicolumn{3}{c}{Solve time (min)}\\
\cmidrule(l){2-4}
\cmidrule(l){5-7}
\cmidrule(l){8-10}
\# PEVs & Min & Avg 	& Max 	& Min 		& Avg 	& Max 	& Min 		& Avg 		& Max \\
\midrule
\midrule
200 	& 8.82	& 10.51	& 12.37	& 1.05		& 1.06	& 1.08	& 0.06		& 0.07		& 0.07 		\\
350		& 2.93	& 3.24	& 3.51	& 1.48		& 1.49	& 1.52	& 1.56		& 6.89	 	& 15.81		\\
500 	& 2.05	& 2.15	& 2.24	& 1.85		& 1.93	& 2.62	& 15.21$^*$	& 65.10$^*$  & 262.81$^*$	\\
700 	& 1.48	& 1.54	& 1.61	& 2.43		& 2.44	& 2.48	& --		& --		& --		\\
1000 	& 1.01	& 1.05	& 1.10	& 3.24		& 3.26	& 3.28	& --		& --		& --		\\
1500 	& 0.65	& 0.68	& 0.72	& 4.72		& 4.74	& 4.81	& --		& --		& --		\\
2000 	& 0.45	& 0.50	& 0.53	& 6.19		& 6.21	& 6.23	& --		& --		& --		\\
5000 	& 0.12	& 0.15	& 0.20	& 14.88		& 14.90	& 14.95	& --		& --		& --		\\
7000 	& 0.09	& 0.10	& 0.12	& 20.59		& 20.77	& 21.94	& --		& --		& --		\\
10000 	& 0.06	& 0.07	& 0.07	& 29.34		& 29.39	& 29.58	& --		& -- 		& --		\\
\bottomrule
\multicolumn{10}{l}{\small{(*) failed to solve two instances (out of memory)}}\\
\multicolumn{10}{l}{\small{(--) out of memory before attaining the desired optimality gap}}
\end{tabular}
\end{center}
\caption{Charging and V2G.}
\label{tab:EV_V2G_result}
\end{table*}

\newpage
\bibliographystyle{alpha}
\bibliography{mybibliography}

\end{document}